\documentclass[leqno]{siamart1116}

\usepackage{graphics,graphicx,epstopdf}

\usepackage{longtable}
\usepackage[tight]{subfigure}
\usepackage{float}
\usepackage{makecell}
\usepackage{comment}
\usepackage{mathtools}

\usepackage[leqno]{amsmath}
\usepackage{amsxtra, amsfonts,amscd,amssymb,graphicx,epsf}
\usepackage[algo2e, vlined,ruled]{algorithm2e}

\newcommand{\obar}[1]{\overline{#1}}
\newcommand{\ubar}[1]{\underline{#1}}

\numberwithin{equation}{section}

\definecolor{am}{RGB}{0,101,189}

\newtheorem{remark}[theorem]{Remark}

\newtheorem{assumption}[theorem]{Assumption}

\newcommand{\st}{\mbox{ s.t.}}

\newcommand{\tr}{\mathrm{tr}}

\newcommand{\R}{\mathbb{R}}

\newcommand{\C}{\mathbb{C}}

\newcommand{\bx}{\bar x}

\newcommand{\M}{\mathcal{M}}

\newcommand{\T}{\mathcal{T}}

\newcommand{\N}{\mathbb{N}}

\newcommand{\kH}{\kappa_{H}}

\newcommand{\kg}{\kappa_g}
\newcommand{\cE}{\mathcal E}

\newcommand{\xb}{\mathbf{x}}

\newcommand{\half}{\frac{1}{2}}

\newcommand{\iprod}[2]{\left \langle #1, #2 \right \rangle }

\newcommand{\be}{\begin{equation}}
\newcommand{\ee}{\end{equation}}
\newcommand{\bee}{\begin{equation*}}
\newcommand{\eee}{\end{equation*}}
\newcommand{\bea}{\begin{eqnarray}}
\newcommand{\eea}{\end{eqnarray}}
\newcommand{\beaa}{\begin{eqnarray*}}
\newcommand{\eeaa}{\end{eqnarray*}}

\newcommand{\Optman}{ARNT}

\usepackage{xcolor}
\usepackage[normalem]{ulem}

\newcommand{\ddt}{\frac{D}{d t} \frac{d}{d t}}

\newcommand{\Hess}{\mathrm{Hess\!\;}}
\newcommand{\grad}{\mathrm{grad\!\;}}
\newcommand{\diag}{\mathrm{diag}}
\newcommand{\rank}{\mathrm{rank}}
\newcommand{\dist}{\mathrm{dist}}
\newcommand{\conv}{\mathrm{conv}}

\DeclareMathOperator*{\argmin}{arg\,min}

\begin{document}

\title{Adaptive Regularized Newton Method for Riemannian Optimization}

\author{
  Jiang Hu\thanks{Beijing International Center for Mathematical Research, Peking University, CHINA
    (\email{jianghu@pku.edu.cn}).}
  \and
  Andre Milzarek\thanks{Beijing International Center for Mathematical Research, Peking University, CHINA
      (\email{andremilzarek@bicmr.pku.edu.cn}).}
  \and
  Zaiwen Wen\thanks{Beijing International Center for Mathematical
Research, Peking University, CHINA (\email{wenzw@pku.edu.cn}).
Research supported in part by NSFC grants 11322109.}
  \and
  Yaxiang Yuan\thanks{State Key Laboratory of Scientific and Engineering Computing, Academy of Mathematics and Systems Science, Chinese Academy of Sciences,
China (\email{yyx@lsec.cc.ac.cn}). Research supported in part by NSFC grants 11331012 and 11461161005.}
}

%\author{Internal Discussion Report\footnotemark[2]}
%\author{Jiang Hu\footnotemark[1] \and Zaiwen Wen\footnotemark[2]  \and Yaxiang Yuan\footnotemark[3]}

%\footnotetext[1]{Beijing International Center for Mathematical
%Research, Peking University, CHINA (wenzw@pku.edu.cn).
%Research supported in part by NSFC grants 11322109.}

%\renewcommand{\thefootnote}{\arabic{footnote}}
\maketitle

\begin{abstract} Optimization on Riemannian manifolds widely arises in eigenvalue computation, density functional theory, Bose-Einstein condensates, low rank nearest correlation, image registration, and signal processing, etc. We propose an adaptive regularized Newton method which approximates the original objective function by the second-order Taylor expansion in Euclidean space but keeps the Riemannian manifold constraints. The regularization term in the objective function of the subproblem enables us to establish a Cauchy-point like condition as the standard trust-region method for proving global convergence.  The subproblem can be solved inexactly either by first-order methods or a modified Riemannian Newton method. In the later case, it can further take advantage of  negative curvature directions. Both global convergence and superlinear local convergence are guaranteed under mild conditions. Extensive computational experiments and comparisons with other state-of-the-art methods indicate that the proposed algorithm is very promising. \end{abstract}


%\pagestyle{myheadings} \thispagestyle{plain} \markboth{z. WEN,}{Adaptive Regularized SCF with Exact Hessian for Electronic Structure Calculation}

%\section*{Outline}
%\begin{itemize}
%  \item  first-order method with nonmonotone line search BB step size
%  \item second-order method, Quasi-Newton method
%  \item Global Convergence Analysis
%  \item Local Convergence Rates
%  \item Numerical Experiments
%\end{itemize}

\section{Introduction}
We consider minimization problems on a Riemannian manifold of the form:
\be \label{prob} \min_{x\in \M} \quad f(x),\ee
where $\M$ is a Riemannian submanifold of an Euclidean space $\cE$ and $f : \M \to \R$ is a smooth real-valued function on $\M$. This problem widely exists in eigenvalue decomposition
\cite{opt-manifold-book}, density functional theory \cite{wen2013adaptive},
Bose-Einstein condensates \cite{wu2015regularized}, low rank nearest correlation
matrix completion \cite{vandereycken2013low}, and many other varieties of applications.
%The eigenvalue problem is to minimize the Rayleigh quotient on the Stiefel manifold which is consisted with the orthogonal matrices. Also, the Riemannian optimization problem is widely used in singular value decomposition, matrix approximations, independent component analysis, density functional theory and Bose Einstein condensates.

Riemannian optimization has been extensively studied over decades of years.
Since problem \cref{prob} can be
 viewed as a general nonlinear optimization problem with constraints, many standard
algorithms \cite{wright1999numerical} can be applied to it directly. These algorithms
 may not be efficient since they do not
utilize the intrinsic structure of the manifold. A first and basic class of manifold optimization methods can be obtained
via modifying and transfering the
nonlinear programming approaches to the manifold setting. In particular,
by performing curvilinear search along the geodesic, Gabay \cite{gabay1982minimizing},
Udri\c{s}te et al.~\cite{udriste1994convex}, Yang \cite{yang2007globally} and Smith
et al.~\cite{smith1994optimization} propose globally convergent steepest descent,
Newton, quasi-Newton and trust-region methods, respectively. Because the computation of
the geodesic may be difficult and expensive, Absil et al.~\cite{opt-manifold-book,
absil2012projection} develop a first-order approximation called retraction to the geodesic.  The previously mentioned
algorithms can be generalized by replacing the geodesic by the retraction and  their global and
local convergence properties have been analyzed in
\cite{AbsilBakerGallivan2007, opt-manifold-book}. Qi \cite{Qi11} and Huang et al.~\cite{huang2015riemannian,
huang2015broyden} propose an extensive class of quasi-Newton methods for
Riemannian manifold problems based on retractions and vector transport. In
\cite{opt-manifold-book}, a nonlinear conjugate gradient
method for Riemannian manifold problems is presented. Bart \cite{vandereycken2013low} and Kressner et
al.~\cite{kressner2014low} show that algorithms using the geometry of a manifold
can be efficient on a large variety of applications. Boumal et
al.~\cite{boumal2016global} establish global convergence rates  for
optimization methods on manifolds.  Moreover, a selection of Riemannian
first-order and second-order methods has been implemented in the software
package Manopt
\cite{manopt}.

 Optimization over the Stiefel manifold (i.e., problems with orthogonality constraints)
 is an
 important special
 case of Riemannian optimization. Edelman et al.~\cite{edelman1998geometry}
 analyze the geometry of this manifold and propose Newton and conjugate gradient
 methods along the geodesic. From the perspective of Euclidean constrained
 optimization, Wen et al.~\cite{wen2013feasible} propose a constraint-preserving
 algorithm on the Stiefel manifold. Jiang et al.~\cite{jiang2015framework} further
 extend their methods and construct a generalized framework. Gao et
 al.~\cite{gao2016new} propose a gradient-type  and column-wise block coordinate
 descent algorithm. Lai et al.~\cite{lai2014folding} study a folding-free global
 conformal mapping for genus-0 surfaces via harmonic energy minimization over
 multiple spheres.
  Zhang et al.~\cite{zhang2014gradient} and Ulbrich et
 al.~\cite{ulbrich2015proximal} present gradient-based algorithms for density
 functional theory which coincide with optimization problems on
 the Stiefel manifold. Wen et
 al.~\cite{wen2013adaptive} develop an adaptively regularized Newton method which
 uses a quadratic
 approximation with exact Euclidean Hessian of the original problem. It often
 exhibits superlinear or quadratic local convergence rate when the subproblem is
 solved accurately. This method has also been extended to 
 Bose-Einstein condensates in \cite{wu2015regularized}.

%Considering the efficiency of the adaptively regularized trust region method, we want to generalize it into a unified framework. This is one of our motivations.

In this paper, we extend the regularized Newton method in
\cite{wen2013adaptive,wu2015regularized} to general Riemannian
optimization problems. Specifically, we approximate problem \cref{prob} and construct a quadratic subproblem by adding a regularization term
to the second-order Taylor expansion of the objective function in Euclidean space. This leads to a class of Euclidean-based model problems that is generally different from classical trust-region-type approaches on Riemannian manifolds \cite{opt-manifold-book}. Typically, the resulting subproblems are easier to be solved than the original problem to a certain extent. We show that, whenever the subproblem can be handled efficiently, a fast
 rate of convergence can be achieved. Since a regularization term is added,
 global convergence can be ensured by adjusting the regularization
 parameters appropriately. %The local superlinear convergence can be established
% as well.  
In fact, convergence can be guaranteed even if the subproblem
 is only solved inexactly as long as it attains a reduction similar to that of a single gradient descent step. Different from
 minimizing the subproblem by the gradient-type methods in
 \cite{wen2013adaptive,wu2015regularized},  we develop a modified Newton
 method using the conjugate gradient method to solve the Newton equation
 followed by a curvilinear search. In particular, our algorithm detects
 directions of negative curvature. We combine them with the previous conjugate
 directions  to construct  new search directions and update the
 regularization parameter based on the negative curvature information. %Because the first conjugate direction we use is the gradient direction, the global convergence can be easily obtained.
  Our extensive numerical experiments show that the proposed method is promising
  and performs comparably well. %Local quadratic convergence is observable when the subproblem is solved sufficiently accurate.

 We should point out that similar second-order type methods have also been developed
 for composite convex
 programs where the objective function is a summation of a smooth function and an
 $\ell_1$-norm or more general convex term. The subproblem in the proximal Newton method by Lee et al.~\cite{lee2014proximal} keeps the $\ell_1$-norm function but approximates the
 smooth part by its  second-order Taylor expansion. A first-order method is then used to solve the resulting proximal subproblem. Byrd et
 al.~\cite{byrd2016inexact} essentially consider the same algorithm but propose a specialized active set strategy to solve the quadratic subproblem. %In their cases, this idea is of great potenial by taking efficient solvers for the subproblems.

This paper is organized as follows. In Section \ref{sec1:firstorderalg}, we review
some preliminaries on Riemannian optimization and present the Riemannian gradient method. The
adaptive regularized Newton method is proposed in Section
\ref{sec:TR-Alg} and its convergence properties are analyzed in Section
\ref{sec:TR-Alg-conv}. Finally, robustness and efficiency of the proposed algorithms are demonstrated
based on several practical examples in Section \ref{sec:num}.

% \footnote[1]{Referring to www.manopt.org}

\subsection{Notation}
 %We use standard linear algebra notation in this paper. 
 Let $(\M, g)$ be a Riemannian manifold. By $\Im_x(\M)$, we denote the set of all
 real-valued functions $f$ defined in a neighborhood of $x$ in $\M$. For a
 given differentiable function $f$ and a point $x \in \M$, $\nabla f(x)$
 ($\nabla^2 f(x)$) and $\grad f(x)$ ($\Hess f(x)$) denote the Euclidean and
 Riemannian gradient (Hessian) of $f$, respectively. Let
 $\iprod{\cdot}{\cdot}~(\| \cdot \|)$ and $
 \iprod{\cdot}{\cdot}_x~(\|\cdot\|_x)$ be the inner product (norm) with
 Euclidean and Riemannian metric, respectively.

\section{Preliminaries on Riemannian optimization} \label{sec1:firstorderalg}
 \label{sec:prelinary}
%%%%%%%%%%%%%%%%%%%%%%%%%%%%%%
% Chapters 3 and 4 in
%\cite{opt-manifold-book}
%\revise{add some basic definitions of manifold, tangent space, retractions,
%optimality conditions}
%%%%%%%%%%%%%%%%%%%%%%%%%%%%%
Many concepts of Riemannian optimization can be regarded as generalizations of the theory
and algorithms from unconstrained Euclidean optimization to problems on
manifolds. A detailed description of the properties of a few commonly used
manifold algorithms are given in \cite{opt-manifold-book}. Here, we only introduce some necessary definitions briefly.

A $d$-dimensional manifold $\M$ is a Hausdorff and second-countable topological space, which is homeomorphic to the $d$-dimensional Euclidean space locally via a family of charts. When the transition maps of intersecting charts are smooth, manifold $\M$ is called a smooth manifold. A function $f$ on $\M$ is said to be $C^k$ at a point $x$ if $f \circ \psi :\psi(U) \subset \R^d \rightarrow \R$ is $C^k$ in which $U$ is an open set in $\M$ containing $x$ and $\psi$ is the mapping defining the chart. A tangent vector $\xi_x$ to $\M$ at $x$ is a mapping such that there exists a curve $\gamma$ on $\M$ with $\gamma(0) = x$, satisfying
\[ \xi_x u := \dot{\gamma}(0) u \triangleq \left. \frac{\mathrm{d}(u(\gamma(t)))}{\mathrm{d}t}\right|_{t=0}, \quad \forall~u \in \Im_x(\M). \]
Then, the tangent space $\T_x\M$ to $\M$ is defined as the set of all tangent vectors to
$\M$ at $x$. %Similar to the Euclidean space, we must introduce the metric on manifold before defining the gradient.
If the manifold $\M$ can be equipped with a smoothly varying inner product $\iprod{\cdot}{\cdot}_x$ between the tangent vectors of the same tangent space, then $\M$ is called a Riemannian manifold. Here, we will always assume that $\M$ is a Riemannian submanifold of an Euclidean space $\cE$, see, e.g., \cite[Section 3.6]{opt-manifold-book} for further details. The norm induced by the Riemannian metric is equivalent to the Euclidean norm, i.e., for all $x \in \M$ there exist parameters $\varpi^m_x,  \varpi^M_x > 0$, which depend continuously  on $x$, such that
\be \label{eq:equi-norm} \varpi^m_x\|\xi\|_x^2 \leq \|\xi\|^2 \leq \varpi^M_x \|\xi\|_x^2, \quad \forall~\xi \in \T_x \M. \ee

The gradient of a real-valued function $f$ on the Riemannian manifold is defined as the unique tangent vector satisfying
\[ \iprod{\grad f(x)}{\xi}_x = Df(x)[\xi] , \quad \forall~\xi \in \T_x\M, \]
where $Df(x)[\xi] = \xi_xf$ and $\grad f(x)$ is called the Riemannian gradient of $f$ at $x$. 
The Riemannian Hessian of $f$ is a linear mapping from $\T_x \M$ to $\T_x\M$ defined by
\[ \Hess f(x)[\xi] = \tilde{\nabla}_{\xi} \grad f(x), \quad \forall~\xi \in \T_x \M, \]
where $\tilde{\nabla}$ is the Riemannian connection which is a unique symmetric affine connection satisfying the Levi-Civita conditions \cite{AbsilBakerGallivan2007}. We refer to \cite{opt-manifold-book} for a more detailed discussion of the Riemannian gradient and Hessian.

First- and second-order optimality conditions for Riemannian
optimization problems take a similar form as standard optimality conditions in the Euclidean space. %If we only have the manifold constraints, the linear independence constraint qualification is naturally satisfied. A detailed
%proof can be found  in \cite{yang2014optimality}.
In particular, let $\M$ be a smooth manifold and let $f : \M \to \R$ be a smooth function on $\M$. Suppose that $x_* \in \M$ is a \textit{stationary point} of problem \cref{prob}, i.e., it holds $\grad f(x_*) = 0$. Furthermore, let $\Hess f(x_*)$ be positive definite on $\T_{x_*}\M$ (w.r.t. the Riemannian metric), then by \cite[Corollary 4.3]{yang2014optimality}, $x_*$ is a strict local solution of \cref{prob}. Analogous second order necessary conditions are presented in \cite{yang2014optimality}.

\subsection{Gradient methods on manifold} \label{sec:grad1} Curvilinear search methods generalize the concept of backtracking line search and gradient descent to the manifold setting and are based on so-called retractions.
%In order to define a curvilinear search method, we need to find a mapping from
%the tangent space onto $\M$. 
A retraction $R$ on $\M$ is a smooth mapping from
the tangent bundle $\T \M := \bigcup_{x \in \M} \T_x\M$ to the manifold $\M$. Moreover, the restriction $R_x$ of $R$ to $\T_x\M$ has to satisfy $R_x(0_x) = x$ and $\mathrm{D} R_x(0_x) =
\mathrm{id}_{\T_x\M}$, where $\mathrm{id}_{\T_x\M}$  is the identity mapping on $\T_x\M$.

Given a retraction $R$, the curvilinear search method computes
\bee \label{eq:ls} x_{k+1} = R_{x_k}(t_k \eta_k),\eee
where $\eta_k \in \T_{x_k}\M$ and $t_k$ is a scalar. Similar to Euclidean
line search methods, $\eta_k$ is chosen as a descent direction and $ t_k$ is a
proper step size determined by either exact or inexact curvilinear search
conditions. %The monotone search ensures the objective function value to decrease monotonically at every iteration.
Given $\rho, \varrho, \delta \in (0,1)$, the monotone and nonmonotone
Armijo rules \cite{ZhaHag04} try to find the smallest integer $h$ satisfying
 \begin{align} \label{eq:MLS-Armijo}
f(R_{x_k}( t_k \eta_k) ) &\le f(x_k) + \rho  t_k \iprod {\grad f(x_k)}
{\eta_k}_{x_k},\\
 \label{eq:NMLS-Armijo}
f(R_{x_k}( t_k \eta_k) ) &\le C_k + \rho  t_k \iprod {\grad f(x_k)} {\eta_k}_{x_k},
\end{align}
respectively, where $t_k = \gamma_k \delta^h$ and $\gamma_k$ is an initial step size.
 Here, the reference
 value $C_{k+1}$ is a convex combination of  $C_k$ and
 $f( x_{k+1})$ and is calculated via $C_{k+1} = (\varrho Q_k C_k +
            f( x_{k+1} ))/Q_{k+1}$,
where $C_0=f(x_0)$, $ Q_{k+1} = \varrho Q_k +1$ and $Q_0=1$.

It is well known that an initial step size computed by the Barzilai-Borwein (BB)
method often speeds up the convergence in Euclidean optimization. Similarly and as in \cite{IanPor17}, we
can consider the following initial step sizes
% to be either $\gamma_k^{(1)} \, \mbox{or} \, \gamma_k^{(2)}$ as: \comm{
\bee  \label{eq:bb} \gamma_k^{(1)} = \frac{\iprod{s_{k-1}}{s_{k-1}}_{x_k}}{|\iprod{s_{k-1}}{v_{k-1}}_{x_k}|} \quad \mbox{ or } \quad
    \gamma_k^{(2)} = \frac{|\iprod{s_{k-1}}{v_{k-1}}_{x_k}|}{
    \iprod{v_{k-1}}{v_{k-1}}_{x_k}}, \eee
 where we can take either
% \bee \label{eq:ebb}  s_{k-1} = x_{k} - x_{k-1}, \quad  v_{k-1} = \grad f(x_k) - \grad f(x_{k-1}). \eee
 %\bea \label{eq:bbe} s_{k-1} = x_k - x_{k-1},\, v_{k-1} = \nabla f(x_k)-\nabla f(x_{k-1}).\eea
%  For Riemannian optimization, the subtraction between $x_{k-1}$ and  $x_k$ is not been defined. In case of manifold $\M$ can be seen as the submanifold of the Euclidean space, we also can do the subtraction and inner product in Euclidean space. Here, we can choose the Euclidean gradient as \eqref{eq:bbe} or the Riemannian gradient as
%Note that the subtractions can be replaced by the corresponding version on
%manifold. It may also be helpful to consider step sizes as
\bee \label{eq:ebb}  s_{k-1} = x_{k} - x_{k-1}, \quad  v_{k-1} = \grad f(x_k) - \grad f(x_{k-1}). \eee
or 
 \bee \label{eq:rbb}  s_{k-1} = - t_{k-1} \cdot {\mathcal T}_{x_{k-1} \rightarrow x_k} ( \grad
 f(x_{k-1})), \quad v_{k-1} = \grad f(x_k) + t_{k-1}^{-1} \cdot s_{k-1}, \eee
and ${\mathcal T}_{x_{k-1} \rightarrow x_k}: \T_{x_{k-1}} \M \mapsto \T_{x_k} \M$ denotes an appropriate vector transport mapping connecting $x_{k-1}$ and $x_k$; see \cite{opt-manifold-book,IanPor17}.
%Then the Riemannian BB method chooses and their Rimannian metrics.
The nonmonotone curvilinear search algorithms using the BB step size
 is outlined in Algorithm \ref{alg:RLSBB}. %\ref{alg:GBB}. % For more details we refer the reader to \cite{GBB-Wen-Yin-2010}.

\begin{algorithm2e}[t]\caption{Riemannian Curvilinear Search Method} \label{alg:RLSBB}
  Input $x_0 \in \M$. Set $k=0$, $\gamma_{\min}\in [0,1], 
  \gamma_{\max}\ge 1$, $C_0 = f(x_0),\, Q_0 = 1$.\\
\While{$\|\grad f(x_k)\| \neq 0$ }
{
  Compute $ \eta_k = -\grad f(x_k)$. \\
  Calculate $\gamma_k$ according to \eqref{eq:bb} and set $\gamma_k=\max(\gamma_{\min}, \min(\gamma_k,\gamma_{\max} ) )$. Then, compute $C_k, \,Q_k$ and find a step size $t_k$ satisfying
  \eqref{eq:NMLS-Armijo}. \\
   Set $x_{k+1}\gets  R_{x_{k}}( t_k \eta_k)$. \\
    Set $k\gets k+1$.
 }
 %Output $x_{k+1}=x_k^{(i)}$.
\end{algorithm2e}

\subsection{Proximal gradient method}
The optimization problem \cref{prob} can also be solved by the proximal gradient method. At the
$k$th iteration, the proximal gradient method linearizes $f(x)$ with a proximal
term to obtain the subproblem
\be \label{eq:sub-L} 
  \min_{x\in \M}~m_k^L(x) = \iprod{\grad f(x_k)} {x-x_k} + \frac{1}{2
  \tau_k} \|x - x_k \|^2, \ee
where $\tau_k$ is the proximal step size and the inner products are defined in Euclidean space. It is easy to see that the solution of \cref{eq:sub-L}, denoted by $x_{k+1}$, is
\be \label{eq:prox-grad}
   x_{k+1}= \mathbf{P}_{\M} (x_k - \tau_k \grad f(x_k) ) = \argmin_{x \in \M}~\frac{1}{2 \tau_k} \|x - x_k + \tau_k \grad f(x_k)\|^2 ,
 \ee
 where $\mathbf{P}_{\M}(x): = \argmin\{\|x - y\| : y\in \M\}$ is the projection
 operator onto $\M$. Notice that $\mathbf{P}_{\M} (x)$ exists if the manifold $\M$ is closed, but it may not be single-valued. If $\M$ is closed and convex, then $\mathbf{P}_{\M} : \R^n \to \M$ defines a function on $\R^n$, see, e.g., \cite{hiriart2013convex}. Furthermore, if $\M$
 is a submanifold of class $C^2$ around $\bx \in \M$,  Proposition 5 in
 \cite{absil2012projection} implies that $R_x(u) = \mathbf{P}_{\M}(x+u)$ is a retraction at $x$ from $\T_x \M $ to $\M$. In this situation, the proximal gradient scheme \cref{eq:prox-grad} can be seen a special case of Algorithm \ref{alg:RLSBB}.
 %Therefore, its convergence can
 %be guaranteed from the properties of Algorithm \ref{alg:RLSBB}.

 % \subsection{Adaptive subgradient (Adagrad) method}
 Recently, Duchi \cite{duchi2011adaptive} propose the so-called Adagrad algorithm to solve online learning and stochastic optimization problems.
 An interesting feature of Adagrad is that it can choose different step sizes for every
 variable. Similarly, we can define an updating formula as
 \be \label{adagrad} \left \{  \begin{aligned}
 	 G_k &= G_{k-1} + \grad f(x_k) \odot \grad f(x_k), \\
 	 x_{k+1} &= \mathbf{P}_{\M} (x_k - \eta
    \grad f(x_k) \oslash {\sqrt{G_k + \epsilon}}),
 \end{aligned} \right.\ee
 where $\eta,\epsilon>0$ and the multiplication ``$\odot$" and division ``$\oslash$" are performed
 component-wise. Other stochastic approaches in deep learning \cite{Goodfellow-et-al-2016} may be applied
 as well. 
 
 \subsection{Convergence of Algorithm \ref{alg:RLSBB}} \label{sec:conv-alg1}
 In this subsection, we give a convergence proof of Algorithm \ref{alg:RLSBB} for the sake of completeness. Our result is a simple generalization of the theory available for monotone line search methods, see, e.g., \cite[Section 4.2]{opt-manifold-book}. Let us also mention that Iannazzo and Porcelli \cite{IanPor17} establish convergence for a similar Riemannian Barzilai-Borwein method using a nonmonotone max-type line search. In the following, the set $\mathcal{L} := \{x \in \M \, : \, f(x) \leq f(x_0)\}$ denotes the level set of $f$ at $x_0$. In comparison to \cite{IanPor17}, our next and first convergence result does not require the level set $\mathcal L$ to be compact.

\begin{theorem} \label{thm:first-order-am}
  Suppose that $f$ is continuously differentiable on the manifold $\mathcal{M}$.  Let $\{x_k\}$ be a sequence generated by Algorithm \ref{alg:RLSBB} using the nonmonotone line search \cref{eq:NMLS-Armijo}.  Then, every accumulation point $x_*$ of the sequence $\{x_k\}$ is a stationary point of problem \cref{prob}, i.e., it holds $\grad f(x_*) = 0$.
  \end{theorem}
\begin{proof} At first, by using $\iprod{\grad f(x_k)}{\eta_k}_{x_k} = - \|\grad f(x_k)\|_{x_k}^2 < 0$ and applying \cite[Lemma 1.1]{ZhaHag04}, it follows $f(x_k) \leq C_k$ and $x_k \in \mathcal L$ for all $k \in \N$. Next, due to
\begin{align*} \lim_{t \downarrow 0} \frac{(f \circ R_{x_k})(t \eta_k) - f(x_k)}{t} - \rho \iprod{\grad f(x_k)}{\eta_k}_{x_k} & \\ &\hspace{-45ex}= \nabla f(R_{x_k}(0))^\top DR_{x_k}(0) \eta_k + \rho \|\grad f(x_k)\|_{x_k}^2 = -(1-\rho) \|\grad f(x_k)\|_{x_k}^2 < 0, \end{align*}
there always exists a positive step size $t_k \in (0,\gamma_k]$ satisfying the monotone and nonmonotone Armijo conditions \cref{eq:MLS-Armijo} and  \cref{eq:NMLS-Armijo}, respectively. Now, let $x_* \in \mathcal M$ be an arbitrary acccumulation point of $\{x_k\}$ and let $\{x_k\}_K$ be a corresponding subsequence that converges to $x_*$. By the definition of $C_{k+1}$ and \cref{eq:MLS-Armijo}, we have
\[ C_{k+1} = \frac{\varrho Q_k C_k + f(x_{k+1})}{Q_{k+1}} < \frac{(\varrho Q_k+1) C_k}{Q_{k+1}} = C_k. \]
Hence, $\{C_k\}$ is monotonically decreasing and converges to some limit $\bar C \in \R \cup \{-\infty\}$. Using $f(x_k) \to f(x_*)$ for $K \ni k \to \infty$, we can infer $\bar C \in \R$ and thus, we obtain
\[ \infty > C_0 - \bar C = \sum_{k=0}^\infty C_k - C_{k+1} \geq \sum_{k=0}^\infty \frac{\rho t_k \|\grad f(x_k)\|_{x_k}^2}{Q_{k+1}}. \]
Due to $Q_{k+1} = 1 + \varrho Q_k = 1 + \varrho  + \varrho^2 Q_{k-1} = ... = \sum_{i=0}^{k} \varrho^i < (1-\varrho)^{-1}$, this implies $\{t_k  \|\grad f(x_k)\|_{x_k}^2\} \to 0$. Let us now assume $\|\grad f(x_*)\| \neq 0$. In this case, we have $\{t_k\}_K \to 0$ and consequently, by the construction of Algorithm \ref{alg:RLSBB}, the step size $\delta^{-1} t_k$ does not satisfy \cref{eq:NMLS-Armijo}, i.e., it holds
\be \label{eq:fst-armijo} - \rho (\delta^{-1}t_k) \|\grad f(x_k)\|_{x_k}^2 < f(R_{x_k}(\delta^{-1}t_k \eta_k)) - C_k \leq f(R_{x_k}(\delta^{-1}t_k \eta_k)) - f(x_k) \ee
for all $k \in K$ sufficiently large. Since the sequence $\{\eta_k\}_K$ is bounded, the rest of the proof is now identical to the proof of \cite[Theorem 4.3.1]{opt-manifold-book}. In particular, applying the mean value theorem in \cref{eq:fst-armijo} and using the continuity of the Riemannian metric, this easily yields a contradiction. We refer to \cite{opt-manifold-book} for more details.
\end{proof} 

Since the iterates generated by Algorithm \ref{alg:RLSBB} stay in the level set $\mathcal L$ (see again \cite{ZhaHag04}), we can derive a slightly stronger convergence result under an additional compactness assumption.

\begin{corollary} Let the sequence $\{x_k\}$ be generated by Algorithm \ref{alg:RLSBB} and let $f$ be continuously differentiable on $\M$. Suppose that the level set $\mathcal L$ is compact. Then, it follows $\lim_{k \to \infty} \| \grad f(x_k)\|_{x_k} = 0$.
\end{corollary}

\begin{proof} The result is a direct consequence of \cref{thm:first-order-am} and of the compactness of $\mathcal L$. Let us also refer to \cite[Corollary 4.3.2]{opt-manifold-book}. \end{proof} 

\section{An Adaptive Regularized Newton Method}
 \label{sec:TR-Alg}

Gradient-type methods often perform reasonably well but might converge slowly when the generated iterates are close
to an optimal solution. Usually, fast local convergence cannot be expected  if only the
gradient information is used, in particular, for difficult non-quadratic or nonconvex
problems.
Starting from an initial point $x_0$, the Riemannian trust-region method
\cite{AbsilBakerGallivan2007,opt-manifold-book} generates the $k$th subproblem as follows
\be \label{subproblem:classic} \begin{aligned}
	\min_{\xi \in \T_{x_k} \M} \quad & \tilde{m}_k(\xi) := f(x_k) + \iprod{\grad f(x_k)}{\xi}_{x_k} + \half \iprod{\Hess f(x_k)[\xi]}{\xi}_{x_k} \\
	\st   \qquad & \iprod{\xi}{\xi}_{x_k} \leq \Delta_k, \\
\end{aligned}\ee
where $\Delta_k$ is the trust-region radius. A common strategy is to apply the
truncated preconditioned conjugate gradient
method (PCG) to solve \cref{subproblem:classic} via the linear system
\be \label{eq:rtr} \grad f(x_k)  +  \Hess f(x_k)[{\xi}]  = 0 \ee
to obtain an approximate (but maybe infeasible) solution $\xi$. The truncated PCG method terminates when
either the residual becomes small enough, a negative curvature direction is detected, or
the trust-region constraint is violated. Then, a trial point is
generated via $z_{k} = R_{x_k}( \xi_k)$ and the new iterate $x_{k+1}$ is set to $z_k$ if a certain reduction condition is satisfied.
Otherwise, the iterate is not updated, i.e., it holds $x_{k+1} := x_k$. Note
that \cref{eq:rtr} differs from the KKT condition for
\cref{subproblem:classic}, since no Lagrange multiplier is involved.

In this paper, we develop an adaptively regularized Riemannian Newton scheme as an alternative approach.
 Specifically, we use a second-order Taylor model to approximate the original objective
 function in the Euclidean space. Moreover, in order to control the definiteness of the model Hessian, a proximal-type penalization is added. The complete objective function of our subproblem is given by
\be\label{eq:mQ}  m_k(x) := \iprod{\nabla f(x_k)}{x - x_k} + \half \iprod {H_k(x-x_k)}
{x-x_k} + \frac{\sigma_k}{2} \|x-x_k\|^2, \ee
where $\nabla f(x_k)$ is the Euclidean gradient and $H_k$ is the Euclidean Hessian of $f$ at $x_k$ or a suitable approximation. % and $\nu$ is either $2$ or $3$.x
The regularization parameter $\sigma_k > 0$ plays a similar role as the
trust-region radius $\Delta_k$ in the trust-region subproblem \cref{subproblem:classic}. %In fact, they act reciprocal to each other in the sense
%that increasing $\sigma_k$ corresponds to decreasing $\Delta_k$ and vice versa. 
A specific choice of $\sigma_k$ will be discussed in subsection \ref{sec:TR-framework}. Our overall idea now is to solve and replace the initial problem \cref{prob} by a sequence of simpler, quadratic subproblems of the form
\be\label{eq:tr-sub}  \begin{aligned} \min_{x \in \M}  \quad & m_k(x),
\end{aligned} \ee
that maintain the manifold constraints. Note that the regularized subproblem \cref{eq:tr-sub} always attains a solution whenever the manifold
$\M$ is compact. % or $H_k$ is positive semidefinite. 

%Let $z_k$ be an approximate solution to \eqref{eq:tr-sub}. The new iteration $x_{k+1}$ is set to $z_k$ if it is a sucessful iteration.  Otherwise, it remains to be $x_k$.
Similar to the classical approaches \cite{NocedalWright06,AbsilBakerGallivan2007,opt-manifold-book}, we embed this basic methodology in a trust-region framework to monitor the acceptance of trial steps and to control the model precision by adjusting the regularization parameter $\sigma_k$. A detailed description of our method can be found in subsection \ref{sec:TR-framework}. Moreover, comparing \cref{eq:tr-sub} and \cref{eq:sub-L}, our approach can also be seen as a hybrid of existing regularized trust-region algorithms \cite{CGT,Qi11} and of the proximal Newton scheme \cite{lee2014proximal} used in convex composite optimization. 

%The subproblem \cref{eq:tr-sub} is actually a reasonable
%choice because the  manifold
%constraints are preserved. An informal
%calculation shows that the
%first-order optimality conditions of \cref{eq:tr-sub} matches that of the original
%problem  \eqref{prob} when the sequence $\{x_k\}$ converges.
In general, we do not need to solve the subproblem \cref{eq:tr-sub} exactly, we only need to find a
point $z_k$  that ensures a
sufficient reduction of the model function $m_k$.
%. Here, we use the condition:
%\be \label{eq:inexact-newton} m_k(z_k) \leq -\alpha \|\grad f(x_k)\|^2 \ee
% for some constant $\alpha$. The condition \cref{eq:inexact-newton} is similar
 For example, as in the classical trust-region method \cite{NocedalWright06}, a fraction of Cauchy decrease condition can be used to guarantee the required model decrease. %Under suitable assumptions, can be satisfied
% by performing one gradient descent step . %this condition \cref{eq:inexact-newton} is often used to ensure the property of global convergence.
 %For certain cases,
 In this respect, the gradient-type methods introduced in subsection \ref{sec:grad1} can be ideal for solving the regularized Newton subproblems
  at the early stage of the algorithm when high accuracy is not needed or when a
good initial guess is not available. Gradient steps can be also useful when the computational cost
of evaluating the Riemannian Hessian is too expensive. When a high accuracy is required,  the
subproblem \cref{eq:tr-sub} can
be solved more efficiently by a single or multiple Riemannian Newton steps as explained in the next
subsection. %The former is similar to  the trust-region algorithm in \cite{AbsilBakerGallivan2007,opt-manifold-book}.
 Together with our specific exploitation of negative curvature information, our
 approach can be a good alternative to the trust-region-type methods 
 \cite{AbsilBakerGallivan2007,opt-manifold-book}.
 %In this case, a major difference to \cite{AbsilBakerGallivan2007,opt-manifold-book} is that 
 % we make use of the negative curvature information from the subproblem.

%{\color{am}ToDo: %Methodology is embedded in a trust-region framework, that controls the model precision by adjusting the regularization parameter $\sigma_k$. Details can be found in subsection 3.2. 
%Essence: how do we solve the subproblem, what's the algorithmic framework, what are the differences to other methods? %Similarity to \cite{lee2014proximal}. Recall motivation: cubic reg. trust-region methods.\cite{Qi11}
%}

\subsection{Solving the Riemannian Subproblem} We use  
an inexact method for minimizing the model \cref{eq:tr-sub} and perform
 a single (or multiple) Riemannian Newton step based on the associated linear system:
\be \label{eq:trsub-nt} \grad m_k(x_k) + \Hess m_k(x_k)[\xi] = 0. \ee
%
%The difference between \eqref{eq:rtr} 
The system \cref{eq:trsub-nt} is solved approximately with 
 a modified conjugate gradient (CG) method up to a certain accuracy.
 Since the model Hessian may be
indefinite, we terminate the CG method when
either the residual becomes small or negative curvature is detected.
Then, a new gradient-related direction is constructed based on the conjugated
directions and a necessary curvilinear search along this direction is utilized
to reach  a sufficient reduction of the objective function. The detailed
procedure is presented in Algorithm \ref{alg:tN}. 

We next discuss a connection between \cref{eq:trsub-nt} and the classical approach \cref{subproblem:classic}--\cref{eq:rtr} in the exact case $H_k = \nabla^2 f(x_k)$. In fact, the definition of the Riemannian gradient implies
\[ \grad m_k(x_k) = \mathbf{P}_{x_k }(\nabla m_k(x_k))  = \mathbf{P}_{x_k } (\nabla f(x_k)) = \grad f(x_k),  \]
where $\mathbf{P}_{x} (u) := \argmin_{v \in {\T_x \M}} \|v -u\|_x$ denotes the orthogonal projection onto ${\T_x \M}$. Using $\nabla^2 m_k(x_k) = \nabla^2 f(x_k) + \sigma_k I$ and introducing the so-called \textit{Weingarten map} ${\mathfrak{W}}_x(\cdot,v) : \T_x\M \to \T_x\M$ for some $v \in \T_x^\bot\M$, it holds
\be \label{eq:hess-formula} \begin{aligned} \Hess m_k(x_k)[\xi] &= \mathbf{P}_{x_k}(\nabla^2 m_k(x_k)[\xi]) + {\mathfrak W}_{x_k}(\xi,\mathbf{P}_{x_k }^\bot(\nabla m_k(x_k))) \\ &= \Hess f(x_k)[\xi] + \sigma_k \xi \end{aligned} \ee
for all $\xi \in {\T_{x_k} \M}$. The Weingarten map ${\mathfrak{W}}_x(\cdot,v)$ is a symmetric linear operator that is closely related to the second fundamental form of $\M$. The projection $\mathbf{P}_{x_k }^\bot$ in \cref{eq:hess-formula} is given explicitly by $\mathbf{P}_{x_k }^\bot = I - \mathbf{P}_{x_k}$. For a detailed derivation of the expression \cref{eq:hess-formula} and further information on the Weingarten map, we refer to \cite{absil2013extrinsic}.

  Although the linear systems \cref{eq:trsub-nt} and
  \cref{eq:rtr} have a similar form, our approach is based on a different model formulation and uses different trial points and reduction ratios. Moreover, inspired by Steihaug's CG method \cite{Ste83} and by related techniques in trust-region based optimization \cite{GouLucRomToi00,opt-manifold-book}, we implement a specific termination strategy whenever the CG methods encounters small or negative curvature. In particular, we utilize the detected negative curvature information to modify and improve our current search direction. 

An overview of the procedure is given in Algorithm \ref{alg:tN}. The method generates two different output vectors $s_k$ and $d_k$, where the vector $d_k$ represents and transports the negative curvature information. 
The new search direction $\xi_k$ is then computed as follows 
 \be \label{eq:subline} \xi_k = \begin{cases} s_k + \tau_k d_k & \text{if } d_k \neq 0, \\ s_k & \text{if } d_k = 0, \end{cases} \quad \text{with} \quad \tau_k := \frac{\iprod{d_k}{\grad m_k(x_k)}_{x_k}}{\iprod{d_k}{\Hess m_k(x_k)[d_k]}_{x_k}}. \ee
 In section \ref{sec:TR-Alg-conv}, we will show that $\xi_k$ is a
 descent direction. %Usually, the relative scaling of $s_k$ and $d_k$ in two directions is hard to handle. But, in our case, 
 Note that the rescaling factor $\tau_k$ in \cref{eq:subline} can be obtained
 without any additional costs. The choice of $\tau_k$ is mainly motivated by our numerical experiments, see also \cref{eq:struc-xik} and \cite{GouLucRomToi00} for a related variant.
 
Once the direction $\xi_k$ is constructed, we carry out a curvilinear
search along $\xi_k$ to generate a trial point $z_k$, i.e., 
\be \label{eq:curvilinear} z_{k} = R_{x_k}(\alpha_k \xi_k ). \ee
The step size $\alpha_k = \alpha_0 \delta^h$ is again chosen by the (monotone) Armijo rule such that $h$ is the smallest integer satisfying
 \be \label{eq:sub-Armijo}
m_k(R_{x_k}( \alpha_0 \delta^h \xi_k) ) \leq \rho  \alpha_0 \delta^h \iprod {\grad m_k(x_k)}{\xi_k}_{x_k}, \ee
where $\rho, \delta \in (0,1)$ and $\alpha_0 \in (0,1]$ are given constants. % whose default value is $1$ in our numerical experiments. 

\begin{algorithm2e}[t] \label{alg:tN}
\caption{A Modified CG Method for Solving Subproblem \eqref{eq:tr-sub} }
%Compute  $\grad m_k(x_k)$ and $\Hess m_k(x_k)$. 
\lnlset{alg:cg-s0}{S0} Set $T > 0$, $\theta > 1$, $\epsilon \geq 0$, $\eta_0 = 0$,  $r_0 = \grad
m_k(x_k)$,  $p_0 = -r_0 $, and  $i = 0$.
			\\
			\While{ $i \leq  n-1$ }
			{	 	
\lnlset{alg:cg-s1}{S1} Compute $\pi_i = \iprod{p_i}{\Hess m_k(x_k)[p_i]}_{x_k}$.\\
\lnlset{alg:cg-S2}{S2} \If {$\pi_i \, /  \iprod{p_i}{p_i}_{x_k} \leq \epsilon$}			
			       {
                   \lIf {$i = 0$}{set $s_k = -p_0, \, d_k = 0$}
			       \lElse{
			       set $s_k = \eta_i, $ 
			       
			       \lIf{$\pi_i \, /  \iprod{p_i}{p_i}_{x_k} \leq -\epsilon$}{$d_k = p_i$, set $\sigma_{est} = |\pi_i| \, /  \iprod{p_i}{p_i}_{x_k} $}
			       \lElse{ $d_k = 0$}
			        break}}
			
\lnlset{alg:cg-s3}{S3} Set $\alpha_i = \iprod{r_i}{r_i}_{x_k}  / \, \pi_i, \, \eta_{i+1} = \eta_i
                   + \alpha_i p_i$, and $r_{i+1} = r_i + \alpha_i\Hess m_k(x_k)[p_i]$.\\
\lnlset{alg:cg-s4}{S4}	\If {$\|r_{i+1}\|_{x_k} \leq \min\{\|r_0\|_{x_k}^\theta, T\}$}
                  {choose $ s_k = \eta_{i+1}, d_k =0$;
				  break;}
\lnlset{alg:cg-s5}{S5} Set $\beta_{i+1} = \iprod{r_{i+1}}{r_{i+1}}_{x_k} /  \iprod{r_i}{r_i}_{x_k}$ and $p_{i+1} = - r_{i+1} + \beta_{i+1} p_i$.\\
			      $i\gets i+1$.
			}
\lnlset{alg:cg-s6}{S6}Update $\xi_k$ according to \eqref{eq:subline}. \\

\end{algorithm2e}

%%%%%%%%%%%%%%%%%%%%%%%%%%%%%%%%%%%%%%

\subsection{The Algorithmic Framework} \label{sec:TR-framework}
We now present our  regularized Newton framework starting from a feasible
initial point $x_0$ and a regularization parameter $\sigma_0$. As described in the last section, the algorithm first computes a trial point $z_k$ to approximately solve the regularized subproblem \cref{eq:tr-sub}. In order to decide whether $z_k$ should be accepted as the next iterate and whether the regularization parameter $\sigma_k$ should be updated or not, we calculate the ratio between the actual reduction of the objective function
   $f(x)$ and the predicted reduction:
  \be \label{alg:TR-ratio} \rho_k = \frac{ f(z_k) - f(x_k)  } {
 m_k(z_k) }.  \ee
   If $\rho_k \ge \eta_1 > 0$, then the iteration is successful and we set
$x_{k+1}= z_k$; otherwise, the iteration is not successful and we set $x_{k+1}=
x_k$, i.e., we have
\be \label{eq:update-x}
x_{k+1} = \begin{cases} z_k, & \mbox{ if } \rho_k \ge \eta_1, \\
  x_k, & \mbox{ otherwise}.
\end{cases}
\ee
The regularization parameter $\sigma_{k+1}$ is updated as follows
 \be \label{alg:tau-up} \sigma_{k+1} \in \begin{cases} (0,\gamma_0\sigma_k] & \mbox{  if }
  \rho_k \ge \eta_2, \\ [\gamma_0\sigma_k,\gamma_1 \sigma_k] & \mbox{  if } \eta_1 \leq \rho_k
  < \eta_2, \\ [\gamma_1 \sigma_k, \gamma_2 \sigma_k] & \mbox{  otherwise}, \end{cases} \ee
 where $0<\eta _1 \le \eta _2 <1 $ and $0 < \gamma_0 < 1<  \gamma _1 \le \gamma _2 $. These parameters  determine how aggressively the
 regularization parameter is adjusted when an iteration is successful or unsuccessful.
% \revise{ When an estimation of the absolute value of the negative curvature, denoted by $\sigma_{est}$, is
% available at the $k$-th subproblem, we calculate
% \be \label{alg:tau-up2} \sigma_{k+1}^{new} = \max \{\sigma_{k+1}, \,
% \sigma_{est} + \tilde \gamma \} ,\ee
% with some small $\tilde\gamma \ge 0$. Then, the parameter $\sigma_{k+1}$ is
% reset to $\sigma_{k+1}^{new}$.}
%
The complete regularized Newton method to solve  \cref{prob} %with the inexact conditions \cref{eq:inexact-newton} 
is summarized in Algorithm \ref{alg:TR-ineact-newton}.

\begin{algorithm2e}[t]
\caption{An Adaptive Regularized Newton Method}
\label{alg:TR-ineact-newton}
\lnlset{alg:TRNS0}{S0}Choose a feasible initial point
$x_0 \in \M$ and an initial regularization parameter $\sigma_0>0$.
 Choose $0<\eta _1 \le \eta _2 <1 $, $0 < \gamma_0 < 1<  \gamma _1 \le \gamma _2 $.  Set  $k:=0$.

\While{stopping conditions not met}
{

\lnlset{alg:TRNS1}{S1}Compute a new trial point $z_k$ according to \cref{eq:curvilinear} and \cref{eq:sub-Armijo}. \\
\lnlset{alg:TRNS2}{S2}Compute the ratio $\rho_k$ via \cref{alg:TR-ratio}. \\
\lnlset{alg:TRNS3}{S3}Update $x_{k+1}$ from the trial point $z_k$ based on \cref{eq:update-x}. \\
\lnlset{alg:TRNS4}{S4}Update $\sigma_{k}$ according to \cref{alg:tau-up}. \\ %Reset $\sigma_{k+1}$ to $\sigma_{k+1}^{new}$ if 
%negative curvature is detected. \\
  $k\gets k+1$.
}

\end{algorithm2e}

\section{Convergence Analysis} \label{sec:TR-Alg-conv}
We now analyze the convergence of Algorithm~\ref{alg:TR-ineact-newton} based on
the model \cref{eq:tr-sub}. Let us note that the analysis can be similarly extended to the algorithm
using cubically regularized subproblems as well. In the following, we summarize and present our main assumptions. 

 \begin{assumption} \label{assum:f} Let $\{x_k\}$ be generated by Algorithm \ref{alg:TR-ineact-newton}. We assume:
%(a) The objective function $f(x)$ is smooth. \\
\begin{itemize}
\item[{\rm(A.1)}] The gradient $\nabla f $ is Lipschitz continuous on the convex hull of the manifold $\M$ -- denoted by $\conv(\M)$, i.e., there exists $L_f > 0$ such that
\end{itemize}
\[ \|\nabla f(x) - \nabla f(y)\| \leq L_f\| x -y\|, \quad \forall~x,y \in \conv(\M).\]
\begin{itemize}
\item[\rm{(A.2)}] There exists $\kg > 0$ such that $\|\nabla f(x_k)\|\leq  \kg$ for all $k \in \N.$
\item[{\rm(A.3)}] There exists $\kH > 0$ such that $\|H_k\| \leq \kH$ for all $k \in \N$.
\item[{\rm(A.4)}] The Euclidean and the Riemannian Hessian are bounded, i.e., there exist $\kappa_F$ and $\kappa_R \geq 1$ such that
\end{itemize}
\[ \|\nabla^2 f(x_k)\| \leq \kappa_F \quad \text{and} \quad \|\Hess f(x_k)\|_{x_k} \leq \kappa_R, \quad \forall ~ k \in \N. \]
\begin{itemize}
\item[{\rm(A.5)}] Let $\varpi^m_{x_k}$, $\varpi^M_{x_k}$ be given as in \cref{eq:equi-norm}. Then, suppose there exists $\ubar{\varpi} > 0$, $\obar{\varpi} \geq 1$ such that $\ubar{\varpi} \leq \varpi^m_{x_k}$ and $\varpi^M_{x_k} \leq \obar{\varpi}$ for all $k \in \N$.
\end{itemize}
\end{assumption}

\begin{remark} \label{remark:lev-compact} Suppose that the level set $\mathcal L$ is compact. Then, by construction of Algorithm \ref{alg:TR-ineact-newton}, we have $f(x_{k+1}) = f(x_k) + \rho_k m_k(z_k) \leq f(x_k)$ if iteration $k$ is successful. Due to \cref{eq:update-x}, it follows $x_k \in \mathcal L$ for all $k$ and the sequence $\{x_k\}$ must be bounded. Hence, in this case, the assumptions {\rm(A.2)} and {\rm(A.4)} hold automatically. Furthermore, since the parameters $\varpi_{x_k}^m$, $\varpi_{x_k}^M$, $k \in \N$, depend continuously on $x_k$, assumption {\rm(A.5)} is also satisfied. 
\end{remark}

\begin{remark} \label{remark:def-mk-bound} The bounds in \cref{assum:f} can also be used to derive a bound for $\Hess m_k(x_k)$. In fact, under the conditions {\rm(A.3)}--{\rm(A.5)} and by \cref{eq:hess-formula}, we have
\begin{align*} \iprod{\xi}{\Hess m_k(x_k)[\xi]}_{x_k} &=  \iprod{\xi}{\Hess f(x_k)[\xi] + \mathbf{P}_{x_k}((H_k - \nabla^2 f(x_k))[\xi])}_{x_k} + \sigma_k \|\xi\|^2_{x_k} \\ & \leq (\kappa_R + (\varpi_{x_k}^M)^\half(\varpi_{x_k}^m)^{-\half}(\kH+\kappa_F) + \sigma_k) \|\xi\|_{x_k}^2 \end{align*}
where we used the linearity and nonexpansiveness of the operator $\mathbf{P}_{x_k}$. In the following, we set $\kappa^M_{x_k} := \kappa_R + (\varpi_{x_k}^M)^\half(\varpi_{x_k}^m)^{-\half}(\kH+\kappa_F)$.
\end{remark}

%%%%%%%%%%%%%%%%%%%%
%
\subsection{Analysis of the Inner Subproblem} At first, we briefly discuss
several useful properties of the modified CG method. %In \cref{lemma:cg-prop}, we also consider the model function $\hat m_k$ which is used in standard Riemannian trust-region methods \cite{AbsilBakerGallivan2007,opt-manifold-book}.
\begin{lemma} \label{lemma:cg-prop} Let the sequences $\{p_i\}_{i=0}^\ell$, $\{r_i\}_{i=0}^\ell$, $\{\eta_i\}_{i=0}^\ell$, and the direction $\xi_k$ be generated by Algorithm \ref{alg:tN}. Then, we have: \vspace{0.5ex}
\begin{itemize}
\item[{\rm(i)}] For all $j = 1,..., \ell$, it holds $p_j \in \T_{x_k}\M$,
\end{itemize}
\be \label{eq:prop-cg} \iprod{p_j}{\Hess m_k(x_k)[p_i]}_{x_k} = 0, \quad \text{and} \quad \iprod{r_j}{r_i}_{x_k} = 0, \quad \forall~i = 0,...,j-1. \ee
\begin{itemize}
\item[{\rm(ii)}] The sequence $\{\tilde m_k(\eta_i)\}$ is strictly decreasing and it holds $\tilde m_k(\xi_k) < \tilde m_k(\eta_\ell)$.\vspace{0.5ex}
\item[{\rm(iii)}] The sequence $\{\|\eta_i\|_{x_k}\}$ is strictly increasing and it holds $ \|\xi_k\|_{x_k} \geq \|\eta_\ell\|_{x_k}$.
%Let $\Hess m_k(x_k)$ be positive definite and suppose that Algorithm \ref{alg:tN} is applied with $\epsilon = \kappa = 0$.  Then, the modified CG method terminates after at most $n$ steps with $\xi_k = \eta_{n-1}$ being a solution of \cref{eq:trsub-nt}.
\end{itemize}
\end{lemma}
\begin{proof} Except for step \ref{alg:cg-S2}, Algorithm \ref{alg:tN} coincides with the standard CG method applied to the quadratic problem $\min_\xi \tilde m_k(\xi)$. 
%
%\[ {\textstyle \min_\xi~\tilde m_k(\xi) := f(x_k) + \iprod{\grad m_k(x_k)}{\xi}_{x_k} + \half \iprod{\xi}{\Hess m_k(x_k)[\xi]}_{x_k}. } \] 
%
Since $\Hess m_k(x_k)$ is a linear operator from $ \T_{x_k} \M$ to $ \T_{x_k} \M$, all iterates generated by Algorithm \ref{alg:tN} will stay in the tangent space $\T_{x_k} \M$. Furthermore, since the Riemannian Hessian is symmetric with respect to the metric $\iprod{\cdot}{\cdot}_{x_k}$, \cite[Proposition 5.5.3]{opt-manifold-book}, part (i) and (ii) essentially follow from the properties of the CG method in Euclidean space. We refer to \cite[Section 5.1]{NocedalWright06} for further details. If $d_k \neq 0$, the estimate $\tilde m_k(\xi_k) \leq \tilde m_k(\eta_\ell)$ follows from \cref{eq:prop-cg} and $\pi_\ell < 0$. The first claim in (iii) is proven in \cite[Theorem 7.3]{NocedalWright06}. To verify $\|\xi_k\|_{x_k} \geq \|\eta_\ell\|_{x_k}$, we first show
\be \label{eq:ri-pi} \|r_i\|_{x_k}^2 = - \iprod{\grad m_k(x_k)}{p_i}_{x_k}, \quad \forall~i = 0,..., \ell-1 \ee
by induction. For $i = 0$, \cref{eq:ri-pi} is obviously satisfied by definition of $r_0$ and $p_0$. Now, let us suppose that \cref{eq:ri-pi} holds for $i = \ell-1$. Then, by \cref{eq:prop-cg} we have
\[ -\iprod{\grad m_k(x_k)}{p_{\ell}}_{x_k} = \iprod{r_0}{r_\ell - \beta_\ell p_{\ell -1}}_{x_k} = - \frac{\|r_\ell\|^2_{x_k}}{\|r_{\ell-1}\|^2_{x_k}} \iprod{r_0}{p_{\ell -1}}_{x_k} = \|r_\ell\|^2_{x_k}. \]
Thus, if $d_k \neq 0$, this implies
\be \label{eq:struc-xik} \xi_k = \eta_\ell + \tau_k d_k = \sum_{i=0}^{\ell-1} \alpha_i p_i - \frac{\|r_\ell\|^2_{x_k}}{\pi_\ell} p_\ell = \sum_{i=0}^\ell |\alpha_i| p_i. \ee
Consequently, since $\xi_k$ and $\eta_\ell$ coincide in the case $d_k = 0$, the estimate $\|\xi_k\|_{x_k} \geq \|\eta_\ell\|_{x_k}$ again follows from \cite[Theorem 7.3]{NocedalWright06} (and from the special structure of $\xi_k$). 
  \end{proof}

We now prove that the direction $\xi_k$ is a descent direction.
\begin{lemma} \label{lemma:descentdir} Let $\{ \alpha_i \}, \{\pi_i\}, \{ p_i\}$, and $\{ \eta_i \}$ be generated by Algorithm \ref{alg:tN} and suppose that the conditions {\rm(A.3)}--{\rm(A.4)} are satisfied. Then, the direction $\xi_k$ -- given in \cref{eq:subline} -- is a descent direction and it holds
\be \label{cond:descent} \frac{\iprod{\grad m_k(x_k)}{\xi_k}_{x_k}}{\|\grad m_k(x_k)\|_{x_k}\|\xi_k\|_{x_k}} \leq - \min \left\{\frac{\epsilon}{2},1\right\} \frac{1}{n(\kappa^M_{x_k}+1)} =: - \lambda_k. \ee
\end{lemma}
\begin{proof} We first analyze the case where Algorithm 2 detects a small or negative curvature and terminates in step \ref{alg:cg-S2}. In this situation, we have
\[ \xi_k = s_k + \tau_k d_k = \begin{cases} - \grad m_k(x_k) & \text{if } \ell = 0 \;\; \text{and} \;\;  \pi_\ell \, \leq \epsilon \|p_\ell\|_{x_k}^{2}, \\ \eta_\ell & \text{if } \ell > 0 \;\; \text{and} \; |\pi_\ell|   \leq \epsilon \|p_\ell\|_{x_k}^{2}, \\ \eta_\ell + \tau_k p_\ell & \text{if } \ell > 0 \;\; \text{and} \;\; \pi_\ell  \, < -\epsilon \|p_\ell\|_{x_k}^{2} \end{cases} \]
with $ \tau_k =  \pi_\ell^{-1} \iprod{\grad m_k(x_k)}{p_\ell}_{x_k}$. We note that condition \cref{cond:descent} is obviously satisfied with $\lambda_k := 1$ in the case $\ell = 0$. Next, let us consider the case $\ell > 0$ and $\pi_\ell < -\epsilon \|p_\ell\|_{x_k}^2$. Due to \cref{eq:ri-pi}, we have
\[ \iprod{\grad m_k(x_k)}{\eta_\ell}_{x_k} = \sum_{i=0}^{\ell-1} \alpha_i \iprod{\grad m_k(x_k)}{p_i}_{x_k} = - \sum_{i=0}^{\ell-1} \frac{\iprod{\grad m_k(x_k)}{p_i}_{x_k}^2}{\pi_i} \]
and thus
\[ \iprod{\grad m_k(x_k)}{\xi_k}_{x_k} = - \sum_{i=0}^\ell \frac{\iprod{\grad m_k(x_k)}{p_i}_{x_k}^2}{|\pi_i|} \leq - \frac{\|p_0\|^4_{x_k}}{\pi_0} \leq - \frac{\|\grad m_k(x_k)\|_{x_k}^2}{\kappa^M_{x_k} + \sigma_k}, \]
where we used the conditions (A.3)--(A.4), \cref{remark:def-mk-bound}, and $\pi_i > 0$, $i = 0,...,\ell-1$. 
By construction of the algorithm, it holds $|\pi_i| = \pi_i > \epsilon \|p_i\|_{x_k}^2$ for all $i = 0,..., \ell-1$ and $|\pi_\ell| = - \pi_\ell > \epsilon \|p_\ell\|^2_{x_k}$. Hence, we obtain
\begin{align*} \| \xi_k \|_{x_k}  & \leq \sum_{j = 0}^{\ell} \frac{|\iprod{\grad m_k(x_k)}{p_i}_{x_k}|}{|\pi_i|} \cdot \|p_i\|_{x_k} \\ & \leq (\ell+1) \|\grad m_k(x_k)\|_{x_k} \cdot \max_{i \in \{0,...,\ell\}} \frac{\|p_i\|_{x_k}^2}{|\pi_i|} \leq \frac{n}{\epsilon} \cdot  \|\grad m_k(x_k)\|_{x_k}. \end{align*} 
Moreover, if $\sigma_k \geq \kappa^M_{x_k} + 2$, then we have $|\pi_i| \|p_i\|_{x_k}^{-2} \geq \sigma_k - \kappa^M_{x_k} > 0$ and the last estimate becomes $\|\xi_k\|_{x_k} \leq n(\sigma_k - \kappa^M_{x_k})^{-1}  \|\grad m_k(x_k)\|_{x_k}$. Combining these results, we now get
\[ \frac{\iprod{\grad m_k(x_k)}{\xi_k}_{x_k}}{\| \grad m_k(x_k) \|_{x_k}\| \xi_k \|_{x_k}} \leq - \min \left\{\frac{\epsilon}{2}, 1\right\} \frac{1}{n(\kappa^M_{x_k} + 1)}. \]
Due to the special structure of $\xi_k$ (see again \cref{eq:struc-xik}), the same estimates can also be used and derived in the remaining cases. This finishes the proof.
\end{proof}

In the next lemma, we prove that the descent property of $\xi_k$ can be carried over to the Euclidean model $m_k$ using the smooth retraction $R$ and that a sufficient reduction of the objective function $m_k$ in the sense of \cref{eq:sub-Armijo} can be ensured.  

\begin{lemma} \label{lemma:desc}
Suppose that the assumptions {\rm(A.2)}--{\rm(A.4)} are satisfied. Let $\rho \in (0,1)$ be arbitrary and set $z_k(t) := R_{x_k}(t \xi_k)$. Then, we have
\be \label{eq:sub-desc} m_k(z_k(t)) \leq  \rho t \iprod{\grad m_k(x_k)}{\xi_k}_{x_k}, \quad \forall~t \in [0, \zeta_k],\ee
where 
\be \label{eq:zeta}  \zeta_k := \min \left \{ (\varpi^M_{x_k})^{-1}, 1 \right \} \min\left \{ \frac{\chi}{\|\xi_k\|_{x_k}}, \frac{2(1-\rho)\lambda_k}{(\kappa_2 \kg + \kappa_1^2 (\kappa_H + \sigma_k))}\frac{\|\grad f(x_k)\|_{x_k}}{\|\xi_k\|_{x_k}} \right \}  \ee
and $\kappa_1, \kappa_2$, $\chi$ are constants that do not depend on $x_k$. 

\end{lemma}
\begin{proof} Let us set $\phi(t) := m_k(R_{x_k}(t \xi_k))$. Then, since $\M$ is an embedded submanifold and using the properties of the retraction $R_{x_k}$, it follows
\begin{align*} m_k(R_{x_k}(t \xi_k)) & = \phi(0) + t \phi^\prime(0) + \int_{0}^t \phi^\prime(s) - \phi^\prime(0) \, \text{ds}
\\ & = t \iprod{\nabla f(x_k)}{\xi_k} + \int_0^t \iprod{\nabla f(x_k)}{(DR_{x_k}(s\xi_k) - {\text{id}})[\xi_k]}  \\ & \hspace{10ex}+ \iprod{(R_{x_k}(s\xi_k) - x_k)}{(H_k + \sigma_k I)[DR_{x_k}(s \xi_k) \xi_k]} \, \text{ds}, \end{align*}
where $\text{id} \equiv \text{id}_{\T_x\M}$ denotes the identity mapping on $\T_x\M$. As in \cite[Section B]{BouAbsCar16}, we define the compact set $K_\chi := \{\xi \in \T\M : \|\xi\| \leq \chi\}$. The smoothness of $R$ now implies
\be \label{eq:retra-bound} \|R_{x_k}(\xi) - x_k\| \leq \int_{0}^1 \|DR_{x_k}(s\xi)[\xi]\| \, \text{ds} \leq \max_{y \in K_\chi} \|DR(y)\| \|\xi\| \ee 
and
\be \label{eq:retra-bound-2} \|DR_{x_k}(\xi) - \text{id}\| \leq \int_{0}^1 \|D^2 R_{x_k}(s\xi)[\xi]\| \, \text{ds} \leq \max_{y \in K_\chi} \|D^2R(y)\| \|\xi\|  \ee
for all $\xi \in K_\chi$. Setting $\kappa_1 := \max_{y \in K_\chi} \|DR(y)\|$ and $\kappa_2 := \max_{y \in K_\chi} \|D^2R(y)\|$ and using the assumptions (A.2)--(A.3), this yields
\begin{align*} m_k(R_{x_k}(t \xi_k)) & \leq t \iprod{\nabla f(x_k)}{\xi_k} + \int_{0}^t (\kappa_2 \kappa_g + \kappa_1^2(\kappa_H + \sigma_k)) s \|\xi_k\|^2 \, \text{ds} \\ & = t \iprod{\grad f(x_k)}{\xi_k}_{x_k} + \half (\kappa_2 \kappa_g + \kappa_1^2(\kappa_H + \sigma_k)) t^2 \|\xi_k\|^2. \end{align*}
if $t \|\xi_k\| \leq \chi$. Thus, by \cref{lemma:descentdir} and setting $\kappa := \kappa_2 \kappa_g + \kappa_1^2(\kappa_H + \sigma_k) $, we obtain
\begin{align*} m_k(R_{x_k}(t \xi_k)) - \rho  t \iprod{\grad m(x_k)}{\xi_k}_{x_k} & \\ & \hspace{-25ex} \leq - (1-\rho) \lambda_k t \|\grad f(x_k) \|_{x_k} \|\xi_k\|_{x_k} + \half \kappa \varpi^M_{x_k }t^2 \|\xi_k\|^2_{x_k}. \\ & \hspace{-25ex} \leq \left [ \half \kappa \varpi^M_{x_k} t - (1-\rho) \lambda_k \frac{\|\grad f(x_k)\|_{x_k}}{\|\xi_k\|_{x_k}} \right ] \cdot t  \|\xi_k\|^2_{x_k}   \end{align*}
if $t \|\xi_k\| \leq \chi$. Finally, using the last estimate, \cref{eq:equi-norm}, and $(\varpi^M_{x_k})^\half \leq \max \{\varpi^M_{x_k},1\}$, this establishes \cref{eq:sub-desc} and \cref{eq:zeta}.
\end{proof}
 
%%%%%%%%%%%%%%%%%
%
\subsection{Global Convergence} \label{sect:globalconver}
%
%%%%%%%%%%%%%%%%%
%
In this section, based on the techniques used in \cite{CGT}, we present global convergence properties of the adaptive regularized Newton method. %Specifically, we will show the iteration sequence $\{x_k\}$ 
%converges to a point of \cref{prob} that satisfies the first-order optimality condition will be showed. The routines are as follows. 
We first investigate the relationship between the reduction ratio $\rho_k$ defined in \cref{alg:TR-ratio}, the regularization parameter $\sigma_k$, and the gradient norm $\|\grad f(x_k)\|_{x_k}$. Under the assumption $\| \grad f(x_k)\|_{x_k} \geq \tau > 0 $, we then derive an upper bound for $\sigma_k$ and show that the iterations will be successful, (i.e, $\rho_k \geq \eta_1$), whenever $\sigma_k$ exceeds this bound. In \cref{thm:convergence} we combine our observations and establish convergence of our method.
% Therefore, we conclude there are only finite number of unsuccessful iterations (i.e, $\rho_k < \eta_1$) due to the upper boundedness
%of $\{ \sigma_k\}$. The global convergence will be obtained immediately.}

The next lemma shows that the distance between $z_k$ and $x_k$ is bounded by some
value related to the regularization parameter $\sigma_k$.
\begin{lemma} \label{lemma:sk-bd} Suppose that the assumptions {\rm(A.2)}--{\rm(A.3)} hold and that $z_k$ satisfies the Armijo condition \cref{eq:sub-Armijo}. Then, it holds
\bee \label{eq:sk-bd} \|z_k - x_k\| \leq \frac{2\kg}{\sigma_k - \kH} \eee
whenever $\sigma_k > \kH$.
\end{lemma}
\begin{proof} By \cref{lemma:descentdir} we have $m_k(z_k) \leq 0$. Thus, it follows
\[ \iprod{\nabla f(x_k)}{z_k -x_k} + \half \iprod{z_k - x_k}{H_k[z_k - x_k]} + \frac{\sigma_k}{2} \|z_k - x_k\|^2 \leq 0. \]
If $\sigma_k \geq \kH$, then the term $\|z_k - x_k\|$ can be bounded by
\[ -\|\nabla f(x_k)\|\|z_k - x_k\| - \frac{1}{2}\kH \|z_k - x_k\|^2 + \half \sigma_k \|z_k -x_k\|^2 \leq 0, \]
and hence,
\[ \|z_k - x_k\| \leq \frac{2\|\nabla f(x_k)\|}{\sigma_k - \kH} \leq \frac{2\kg}{\sigma_k - \kH}.\]
\end{proof}

When the regularization parameter is sufficiently large, our model defines a good approximation of the initial problem \cref{prob}. In this case, a successful iteration and sufficient reduction of the objective function can be ensured.
 
\begin{lemma} \label{lemma:suc-tr}
Suppose that the conditions {\rm(A.1)}--{\rm(A.4)} hold and that $z_k$
satisfies the Armijo condition \cref{eq:sub-Armijo}. Furthermore, let us assume $ g_k := \|\grad f(x_k)\|_{x_k} \neq 0$ and
\bee \label{eq:tr-rad-succond}
\sigma_k \geq \max \left\{ \kappa_{x_k}^M, \kH  + \vartheta_k \max\left\{ {\frac{1}{\sqrt{\chi}}},  {\frac{\sqrt{A_2^k}}{\sqrt{g_k}}}, \frac{A_3^k \vartheta_k}{g_k} \right\}\!\!\right\}, \quad \vartheta_k := \sqrt{\frac{A_1^k\max\{\varpi_{x_k}^M,1\}}{(1-\eta_2) g_k}}
\eee
where $\kappa := \kappa_2\kappa_g + 2 \kappa_1^2 \kappa_H$, $A_1^k := 2 \kappa_g^2 \alpha_0 (\rho\lambda_k\delta)^{-1}(L_f+\kappa_H)$, $A_2^k := ((1-\rho)\lambda_k)^{-1}\kappa $, and $A_3^k := ((1-\rho)\lambda_k)^{-1}\kappa_1^2$.
%$\kappa := \kappa_2^2 + \kappa_1 \kappa_3, \, A_1 := 2(L_f +  \kH)\kg^2,\, A_2 := \rho \lambda \varpi_1,\, A_3 := \frac{\kappa}{2(1 - \rho)\lambda \varpi_1 \|\grad f(x_k)\|^2},$ 
%$A_4 := \max \left\{ \frac{ \kappa_3\kappa_g + \kappa\kH}{2(1 - \rho)\lambda \varpi_1 \|\grad f(x_k)\|^2}, \frac{\hat{\lambda} }{\chi \|\grad f(x_k)\|} \right\}.$
Then, iteration $k$ is very successful, i.e., it holds $ \rho_k \ge \eta_2$  and $\sigma_{k+1} \leq \gamma_0\sigma_k$. 
\end{lemma}
\begin{proof}
Using the Lipschitz continuity of $\nabla f$ and (A.3), it follows
\begin{align*} f(z_k) - f(x_k) - m_k(z_k) 
 %& \hspace{-15ex} = f(z_k) - f(x_k) -\iprod{\nabla f(x_k)}{z_k-x_k}  - \half \iprod {H_k[z_k-x_k]} {z_k-x_k}
 %- \frac{\sigma_k}{2} \|z_k-x_k\|^2 \\
 & =  \iprod{\nabla f(x_k + \tau(z_k -x_k)) -  \nabla f(x_k)}{z_k -x_k}  \\ & \hspace{4ex} - \half \iprod {H_k[z_k-x_k]} {z_k-x_k}
- \frac{\sigma_k}{2} \|z_k-x_k\|^2 \\
& \leq \half (L_f + \kH)\|z_k - x_k\|^2, \end{align*}
for some $\tau \in (0,1)$. Applying \cref{lemma:desc}, \cref{lemma:sk-bd}, and the Armijo condition \cref{eq:sub-Armijo}, we now obtain
 \beaa  1 - \rho_k  &=&  \frac{f(z_k) - f(x_k) - m_k(z_k)}{ - m_k(z_k)} \\
%& \leq &
%\frac{L + \half \kH} {\rho c\|\grad f(x_k)\|^2 \min \left\{ \frac{\vartheta^2}{\mu^2 \kH}, \frac{\vartheta}{\sqrt{\mu^3 \sigma_k \|\grad f(x_k)\|}}, \delta_c,  \frac{\delta_\mu}{\|\grad f(x_k)\|} \right\} }  \cdot \frac{25(\kg + \kH)^2}{\sigma_k^2} \max\{1,\sigma_k\} \\
& \leq &
\frac{(L_f + \kappa_H)\|z_k - x_k\|^2}{2 \rho \lambda_k \alpha_k \| \grad f(x_k) \|_{x_k} \|\xi_k\|_{x_k}} \\ 
& \leq &
\frac{2(L_f + \kappa_H)\kappa_g^2}{\rho \lambda_k \alpha_0^{-1} \delta} \cdot \frac{\max\{\varpi_{x_k}^M,1\}}{(\sigma_k-\kappa_H)^2 g_k}  \cdot \max \left \{ \frac{1}{\chi}, \frac{\kappa_2\kappa_g + \kappa_1^2(\sigma_k + \kappa_H)}{2(1-\rho)\lambda_k g_k} \right \} \\ 
& \leq &
\frac{A_1^k \max\{\varpi_{x_k}^M,1\}}{(\sigma_k-\kappa_H)^2 {g_k}} \max \left\{ \frac{1}{\chi}, \frac{\kappa}{(1-\rho)\lambda_k g_k}, \frac{\kappa_1^2 (\sigma_k - \kappa_H)}{(1-\rho)\lambda_k g_k} \right \} \\ 
& = &
\frac{A_1^k \max\{\varpi_{x_k}^M,1\}}{(\sigma_k-\kappa_H)^2 {g_k}} \max \left \{ \frac{1}{\chi}, \frac{A_2^k}{g_k}, \frac{A_3^k}{g_k}(\sigma_k - \kappa_H) \right\} \\
 & \leq & 1-\eta_2. \eeaa
The above inequality shows $\rho_k \geq \eta_2$. %Moreover, due to $\sigma_k \geq \kappa^M_{x_k}$, no negative curvature is encountered. 
Finally, step \ref{alg:TRNS4} of Algorithm \ref{alg:TR-ineact-newton} implies $\sigma_{k+1} \leq \gamma_0\sigma_k$, as desired.
\end{proof}

We next prove that the regularization parameters can be bounded.
\begin{lemma} \label{lemma:tauk-bd} Suppose that the assumptions {\rm(A.1)}--{\rm(A.5)} are satisfied and there exists $\tau > 0$
such that $\|\grad f(x_k)\|_{x_k} \geq \tau$ for all $k \in \N$. Then, the sequence $\{\sigma_k\}$ is bounded, i.e., there exists $L_\tau \geq 0$ such that
\be \label{eq:tauk-bd} \sigma_k \leq L_\tau, \quad \forall~ k \in \N. \ee %\vspace{-0.8cm}
\end{lemma}
\begin{proof} At first, using the bounds in (A.5), it holds $\kappa_{x_k}^M \leq \kappa_R + (\obar{\varpi})^\half(\ubar{\varpi})^{-\half}(\kH+\kappa_F) =: \bar \kappa_M$. Hence, it follows $\lambda_k \geq \min\left\{\frac{\epsilon}{2},1\right\}(n(\bar\kappa_M+1))^{-1} =: \bar\lambda$ and similarly,
\[ A_1^k \leq (2 \kappa_g^2 \alpha_0(L_f+\kappa_H))(\rho\delta\bar\lambda)^{-1} =: A_1, \quad A_2^k \leq \kappa((1-\rho)\bar\lambda)^{-1} =: A_2, \]
and $A_3^k \leq \kappa_1^2 ((1-\rho)\bar\lambda)^{-1} =: A_3$. We now define
	%Let $A_3 := \frac{\kappa}{2(1 - \rho)\lambda \varpi_1 \epsilon^2},$ 
	%$A_4 := \max \left\{ \frac{ \kappa_3\kappa_g + \kappa\kH}{2(1 - \rho)\lambda \varpi_1 \epsilon^2}, \frac{\hat{\lambda} }{\epsilon \chi } \right\}$. Define
%
\[ \kappa_\tau : = \max\left \{ \bar\kappa_M, \kH + \vartheta_\tau \max\left\{ \frac{1}{\sqrt{\chi}}, \frac{\sqrt{A_2}}{\sqrt{\tau}}, \frac{A_3 \vartheta_\tau}{\tau}\right\} \right\}, \quad \vartheta_\tau := \sqrt{\frac{A_1 \obar{\varpi}}{(1-\eta_2)\tau}}. \]
Let us assume that the bound $\sigma_k  \geq \kappa_{\tau}$ holds for some $k \geq 0$.
Then, \cref{lemma:suc-tr} implies that iteration $k$ is very
successful with $\sigma_{k+1} \leq \sigma_k$. Consequently, when
$\sigma_0 \leq \gamma_2 \kappa_{\tau} $, we have $\sigma_k \leq \gamma_2 \kappa_{\tau}$,
$k \geq 0$, where the factor $\gamma_2$ is introduced to cover
the case that $\sigma_k$ is less than $\kappa_{\tau}$ and iteration $k$
is not very successful. Setting $L_\tau := \max\left \{\sigma_0, \gamma_2
\kappa_{\tau} \right \}$, we obtain \cref{eq:tauk-bd}.
\end{proof}

Based on the results in \cite{CGT}, \cite[Section 7]{opt-manifold-book} and similar to \cite{Qi11}, we now show global convergence of our adaptive regularized Newton method. We first analyze the behavior of Algorithm \ref{alg:TR-ineact-newton} under the assumption that only finitely many successful iterations are performed.
\begin{lemma} \label{lemma:finite} Suppose that the assumption {\rm(A.1)}--{\rm(A.5)} are satisfied
and there are only finitely many successful iterations. Then, it holds
$x_k = x_*$ for all sufficiently large $k$ and $\grad f(x_*) = 0$. \end{lemma}
\begin{proof} Let the last successful iteration be indexed
by $\ell$, then, due to the construction of Algorithm \ref{alg:TR-ineact-newton}, it holds
$x_{\ell +1} = x_{k} = x_*$, for all $k \geq \ell + 1$. Since all iterations $k \geq \ell +1$
are unsuccessful, the regularization parameter $\sigma_k$ tends to infinity as $k \to \infty$.
If $\|\grad f(x_{\ell+1})\|_{x_{\ell+1}} > 0$, then we have  $\|\grad f(x_k)\|_{x_k} = \|\grad f(x_{\ell +1})\|_{x_{\ell +1}} > 0$ for all
$k \geq \ell +1$, and \cref{lemma:tauk-bd} implies that $\sigma_k$ is
bounded above, $k \geq \ell + 1$. This contradiction completes the proof.
 \end{proof}

The following theorem generalizes \cite[Theorem 2.5]{CGT} and represents our main convergence result in this section.
\begin{theorem} \label{thm:convergence} Suppose that the assumptions {\rm(A.1)}--{\rm(A.5)} hold and let $\{f(x_k)\}$ be bounded from below. Then, either
%\vspace{-0.1cm}
\[ \grad f(x_\ell) = 0 \ \ \text{for some} \ \ \ell \geq 0  %\vspace{-0.8cm}
%\[ \mbox{ or } \lim_{k \to \infty} E(x_k) = - \infty \] \vspace{-0.8cm}
\quad \text{or} \quad \liminf_{k \to \infty} \|\grad f(x_k)\|_{x_k} = 0. \] %\vspace{-0.8cm}
\end{theorem}
\begin{proof} Due to \cref{lemma:finite}, we only have to consider
the case when infinitely many successful iterations occur. Let us assume that
there exists $\tau > 0$ such that
\be \label{eq:th-as-grad} \| \grad f(x_k)\|_{x_k} \geq \tau, \quad \forall~k \geq 0 \ee
and let $k \in {\mathcal S}$ with ${\mathcal S} := \{k \in \N : \text{iteration $k$ is successful or very successful}\}$ be given. As in the proof of \cref{lemma:tauk-bd} there exists $\bar \lambda$ such that $\lambda_k \geq \bar \lambda$ for all $k \in \N$. Now, \cref{lemma:desc} and \cref{lemma:tauk-bd} imply
 \begin{align*}
 	f(x_k) - f(z_k)  & \geq  \eta_1 \cdot (- m_k(z_k)) 
 	\geq  \eta_1 \rho \bar \lambda \alpha_0^{-1} \delta \zeta_k \cdot  \|\grad f(x_k)\|_{x_k} \|\xi_k\|_{x_k} \\  & \geq
	\eta_1   \rho \bar\lambda \delta  (\alpha_0\obar{\varpi})^{-1} \tau \cdot \min \left\{ \chi, \frac{2(1-\rho)\bar\lambda \tau}{\kappa + \kappa_1^2(L_\tau - \kappa_H)}  \right\} =: \delta_\tau.
 \end{align*}
Summing up over all iterates yields
\be \label{eq:descent-E} f(x_0) - f(x_{k+1}) = \sum_{j = 0, j \in {\mathcal S}}^k
f(x_j) - f(x_{j+1}) \geq |{\mathcal S} \cap \{1,...,k\}| \cdot \delta_\tau. \ee
Since $\mathcal S$ is not finite, we have $|{\mathcal S} \cap \{1,...,k\}| \to \infty$
as $k \to \infty$. Consequently, inequality \cref{eq:descent-E} implies
$\lim_{k \to \infty} f(x_0)-f(x_{k+1}) =  \infty$ which is  contradiction to the lower boundedness of $\{f(x_k)\}$. Hence,
assumption \cref{eq:th-as-grad} must be false and $\{\|\grad f(x_k)\|_{x_k} \}$ has a subsequence that converges to zero. \end{proof}

\begin{remark}As in \cref{sec:conv-alg1}, it is possible to obtain a slightly stronger result and establish convergence of the full sequence $\|\grad f(x_k)\|_{x_k} \to 0$ as $k \to \infty$. However, this requires additional assumptions on the retraction and on the Lipschitz continuity of the Riemannian gradient. We refer to \cite[Corollary 4.2.1]{Qi11} for a related discussion and result.
\end{remark}

\subsection{Local Convergence}
In this part, we analyze the local convergence properties of Algorithm \ref{alg:TR-ineact-newton}. Because our inner solver is a regularized Newton method, the local superlinear convergence can be established using similar techniques as in the standard trust-region method \cite{wright1999numerical}. Following \cite[Proposition 7.4.5]{opt-manifold-book}, we first present an assumption on the boundedness of the second-order covariant derivatives $\ddt R$ of the retraction $R$. 

\begin{assumption} \label{assum:Lipofgrad} 
Suppose that there exists $\beta_R, \delta_R > 0$ such that
\bee  \left\| \ddt R_x(t \xi) \right \|_{x} \leq \beta_R \eee
for all $x \in \M$, all $\xi \in \T_x \M$ with $\|\xi\|_{x} = 1$ and all $t < \delta_R$.
\end{assumption}
 
We refer to \cite[Chapter 5]{opt-manifold-book} for a detailed discussion of covariant derivatives. Let us note that \cref{assum:Lipofgrad} is satisfied whenever the manifold $\M$ is compact, see, e.g., \cite[Corollary 7.4.6]{opt-manifold-book}.
We now present our main assumptions that are necessary to prove fast local convergence of Algorithm \ref{alg:TR-ineact-newton}. Let us emphasize that our assumptions are similar to the ones used in other Riemannian optimization frameworks.

\begin{assumption} \label{assum:local} Let $\{x_k\}$ be generated by Algorithm \ref{alg:TR-ineact-newton}.We assume:
\begin{itemize}
\item[{\rm(B.1)}] The sequence $\{x_k\}$ converges to $x_*$. \vspace{0.5ex}
\item[{\rm(B.2)}] The Euclidean Hessian $\nabla^2 f$ is continuous on $\conv(\M)$. \vspace{0.5ex}
\item[{\rm(B.3)}] The Riemannian Hessian $\Hess f$ is positive definite at $x_*$ and the constant $\epsilon$ in Algorithm \ref{alg:tN} is set to zero. \vspace{0.5ex}
\item[{\rm(B.4)}] The matrices $H_k$, $k \in \N$, satisfy the following Dennis-Mor\'{e} condition: 
\end{itemize}
\[ \label{eq:hess-dm} \frac{\|(H_k - \nabla^2 f(x_k))[z_k - x_k] \|}{\|z_k - x_k\|} \to 0, \quad \text{whenever} \quad \| \grad f(x_k)\|_{x_k} \to 0, \]
\begin{itemize}
\item[{\rm(B.5)}] $H_k$ is a good approximation of the Euclidean Hessian $\nabla^2 f$, i.e., it holds
\end{itemize}
%
%   \[ \label{eq:hessian} \frac{\|(H_k - \nabla^2 f(x_k))(z_k - x_k) \|}{\|z_k - x_k\|} \to 0, \quad \text{whenever} \quad \| \grad f(x_k)\|_{x_k} \to 0. \]
%
   \[ \label{eq:hessian} \| H_k - \nabla^2 f(x_k) \| \to 0, \quad \text{whenever} \quad \| \grad f(x_k)\|_{x_k} \to 0. \]
%\begin{itemize}
%\item[]with some positive constant $\beta_H$. \vspace{0.5ex}
%\item[{\rm(B.6)}] $\Hess \hat{f}_x$ is continuous at $0_x$ uniformly in $x$ in a neighborhood of $x_*$. \vspace{0.5ex}  
%\end{itemize}
\end{assumption}

In the following lemma and inspired by \cite[Theorem 4.3]{CGT} and \cite[Theorem 4.2.2]{Qi11}, we show that the iterations generated by Algorithm \ref{alg:TR-ineact-newton} are eventually very successful. Due to \cref{alg:tau-up}, this also implies that the sequence of regularization parameters $\{\sigma_k\}$ converges to zero as $k \to \infty$. 

\begin{lemma} \label{lemma:dimin}
Let the conditions {\rm(A.3)} and {\rm(B.1)}--{\rm(B.4)} be satisfied. Then, all iterations are eventually very successful.	
\end{lemma}

\begin{proof} First, \cref{thm:convergence} implies that $x_*$ is stationary point of problem \cref{prob}, i.e., we have $\grad f(x_k) \to \grad f(x_*) = 0$ as $k \to \infty$. Moreover, since $\{x_k\}$ converges to $x_*$, the assumptions (A.2) and (A.4)--(A.5) are satisfied. We next use a connection between $\xi_k$ and $\grad f(x_k)$ that was established in the proof of \cref{lemma:descentdir}; it holds 
\be \label{eq:bound-xikgk} \|\xi_k\|_{x_k} \leq \min\{\epsilon^{-1},1\} n \cdot \|\grad f(x_k)\|_{x_k} \to 0, \quad k \to \infty. \ee
Hence, we have $\|\xi_k\| \leq \chi$ for all $k$ sufficiently large and thus, from \cref{eq:retra-bound} it follows
\be \label{eq:bound-zkxk}  \|z_k-x_k\|  \leq \kappa_1 \alpha_k \|\xi_k\|  \leq \min\{\epsilon^{-1},1\} n \kappa_1 \sqrt{\obar{\varpi}} \cdot \|\grad f(x_k)\|_{x_k}. \ee
Similar to \cite[Section B]{BouAbsCar16} and by combining \cref{eq:retra-bound}--\cref{eq:retra-bound-2}, we obtain
\be \label{eq:est-r1} \|z_k - x_k - \alpha_k \xi_k\| = \|R_{x_k}(\alpha_k \xi_k) - x_k - \alpha_k \xi_k\| \leq \frac{\obar{\varpi}\kappa_2}{2} \alpha_k^2 \|\xi_k\|_{x_k}^2 \ee
for all $k$ sufficiently large. Using the continuity of the Riemannian Hessian and (B.3) there exists $\nu > 0$ such that $ \iprod{\xi}{\Hess f(x_k)[\xi]}_{x_k} \geq \nu \|\xi\|_{x_k}^2$ for all $\xi \in \T_{x_k}\M$ and $k \in \N$ sufficiently large. Setting $m_k^F(x) := m_k(x) - \frac{\sigma_k}{2} \|x-x_k\|^2$, this implies 
%Next, using the Lipschitz continuity of the Hessian, assumption (B.3) and setting $y_k := z_k - x_k$, it follows
%
\[ \iprod{\xi_k}{\Hess m_k(x_k)[\xi_k]}_{x_k} \geq (\nu + \sigma_k) \|\xi_k\|_{x_k}^2 - \left |\iprod{\xi_k}{(\Hess m^F_k(x_k) - \Hess f(x_k))[\xi_k]}_{x_k} \right |. \]
Due to \cref{eq:hess-formula}, we have $(\Hess m^F_k(x_k) - \Hess f(x_k))[\xi_k] = \textbf{P}_{x_k}((H_k - \nabla^2 f(x_k))[\xi_k])$ and thus, it holds
\be \label{eq:est-mf-hess}
\begin{aligned}
\frac{\left |\iprod{\xi_k}{(\Hess m^F_k(x_k) - \Hess f(x_k))[\xi_k]}_{x_k} \right |}{\|\xi_k\|_{x_k}^2} & \\ & \hspace{-30ex} \leq c_1 \frac{\|(H_k - \nabla^2 f(x_k))[z_k-x_k]\|}{\|z_k-x_k\|}\frac{\|z_k-x_k\|}{\alpha_k \|\xi_k\|_{x_k}} + c_2 \frac{\|z_k - x_k - \alpha\xi_k\|}{\alpha_k \|\xi_k\|_{x_k}},
\end{aligned} \ee
where $c_1, c_2 > 0$ are suitable constants that only depend on $\ubar{\varpi}$, $\obar{\varpi}$, $\kappa_H$ and $\kappa_F$. By (B.4), \cref{eq:bound-zkxk}, and \cref{eq:est-r1}, the last term converges to zero as $k \to \infty$. Consequently, we can infer $\iprod{\xi_k}{\Hess m_k(x_k)[\xi_k]}_{x_k} \geq \frac{\nu+\sigma_k}{2} \|\xi_k\|_{x_k}^2 $ for all $k$ sufficiently large. This also implies that Algorithm \ref{alg:tN} does not stop in iteration $i=0$. Hence, applying \cref{lemma:cg-prop} (ii), we obtain
\begin{align*} \iprod{\grad m_k(x_k)}{\xi_k}_{x_k} & \leq \tilde m_k(\eta_1) - f(x_k) - \half \iprod{\xi_k}{\Hess m_k(x_k)[\xi_k]}_{x_k} \\ & \hspace{-7ex}\leq - \half \left ( \frac{\|g_k\|_{x_k}^4}{\iprod{g_k}{\Hess m_k(x_k)[g_k]}_{x_k}} + \frac{\nu+\sigma_k}{2} \|\xi_k\|_{x_k}^2 \right) \leq - \frac{\nu+\sigma_k}{4} \|\xi_k\|_{x_k}^2, \end{align*}
where $g_k := \grad f(x_k)$. Using this estimate in the proof of \cref{lemma:desc}, we can now derive a more refined bound for the step size $\alpha_k$. In particular, it holds
\[ m_k(R_{x_k}(t\xi_k)) \leq - \frac{\rho\nu}{4} t \|\xi_k\|^2_{x_k}, \quad \forall~t \in [0, \bar t], \,\, \text{with} \,\, \bar t:= \frac{1-\rho}{2\obar{\varpi}} \min \left \{ \frac{\nu}{\kappa_2\kappa_g + \kappa_1^2\kappa_H}, \frac{1}{\kappa_1^2} \right \} \]
and thus, we have 
\be \label{eq:bound-mk} - m_k(z_k) \geq \frac{\rho\nu\delta}{4\alpha_0} \bar t \cdot \|\xi_k\|^2_{x_k} \geq \frac{\rho\nu\delta\bar t}{4\alpha_0\kappa_1\obar{\varpi}} \cdot \|z_k-x_k\|^2 =: \bar \delta \cdot  \|z_k-x_k\|^2, \ee 
for all $k$ sufficiently large. Next, applying a second order Taylor expansion, it follows 
\[ f(z_k) - f(x_k) - m_k(z_k) \leq \half \iprod{(\nabla^2 f(x^\delta_k) - H_k)[z_k-x_k]}{z_k-x_k}, \] 
for some suitable $\delta_k \in [0,1]$ and $x^\delta_k := x_k + \delta_k(z_k-x_k)$. Using the continuity of $\nabla^2 f$, (B.4), and the bound \cref{eq:bound-mk}, we finally obtain
\[ 1 - \rho_k \leq \frac{1}{2\bar\delta} \left [ \frac{\|(\nabla^2 f(x_k)-H_k)[z_k-x_k]\|}{\|z_k - x_k\|} + \|\nabla^2 f(x_k^\delta) - \nabla^2 f(x_k)\| \right ] \to 0, \]
as $k \to \infty$. This finishes the proof.
\end{proof}

Next, we establish superlinear convergence of the proposed method. In comparison to \cref{lemma:dimin}, we need a stronger assumption on the matrices $H_k$ to guarantee that the CG method eventually only uses the natural stopping criterion in step \ref{alg:cg-s4}. In the following, let $\hat{f}_x := f \circ R_x$ denote the pullback of $f$ through $R_x$ at $x$ and let $0_x$ be the zero element of $\T_x\M$.

\begin{theorem} \label{thm:mainthm}
Suppose that \cref{assum:Lipofgrad} and the conditions {\rm(B.1)}--{\rm(B.3)} and {\rm(B.5)} are satisfied and let $\alpha_0 = 1$ and $\rho \in (0,\half)$. Then, the sequence $\{x_k \}$ converges q-superlinearly to $x_*$. 
\end{theorem}
\begin{proof}
For convenience, we again set $g_k := \grad f(x_k)$. We further note that the conditions (B.1) and (B.5) imply (A.2)--(A.5). Due to \cref{assum:Lipofgrad} and applying \cite[Proposition 19]{boumal2016global}, the following bound holds for any smooth function $h$ on $\M$
\be \label{eq:apphessofpullpack} \| \Hess h(x) - \Hess \hat{h}_x(0_x)\|_{x} \leq \beta_R \|\grad h(x)\|_{x}, \ee
where the operator norm is induced by the Riemannian metric on $\T_{x} \M$.
%Let $P_k(x):= m_k(x) - \frac{\sigma_k}{2}\|x -x_k\|^2$. 
Similar to the proof of \cref{lemma:dimin} and using (B.3), (B.5), and the uniform estimate
%
% Noticing that $\grad P_k(x) = g_k$ and $\nabla^2 P_k(x_k) = H_k$, we have 
\be \label{eq:hess-mf-m}  | \iprod{\xi}{(\Hess m_k^F(x_k) - \Hess f(x_k))[\xi]}_{x_k} | \leq c \cdot  \| H_k - \nabla^2 f(x_k) \| \|\xi\|_{x_k}^2, \ee
for $\xi \in \T_{x_k}\M$ and for some constant $c > 0$, we can infer that $\Hess m_k(x_k)$ is positive definite for all $k$ sufficiently large. Thus, the structure of Algorithm \ref{alg:tN} now implies
\be \label{eq:tcg-stop} \| g_k + \Hess m_k(x_k)[\xi_k] \|_{x_k} \leq \| g_k \|^\theta_{x_k}, \quad \theta > 1.   \ee
Also, by \cref{lemma:dimin}, we have $\sigma_k \to 0$ as $k \to \infty$. Hence, there exists $\bar \sigma$ such that $\sigma_k \leq \bar \sigma$ for all $k \in \N$. We next show that the full step size $\alpha_k = 1$ satisfies the Armijo condition \cref{eq:sub-Armijo} whenever $k$ is sufficiently large. First, by \cref{lemma:cg-prop} (ii) and \cref{remark:def-mk-bound}, we have
\be \label{eq:xik-low-bound} \|\xi_k\|_{x_k} \geq \|\eta_1\|_{x_k} = \frac{\|g_k\|^3_{x_k}}{\iprod{g_k}{\Hess m_k(x_k)[g_k]}_{x_k}} \geq \frac{\|g_k\|_{x_k}}{\kappa_{x_k}^M + \sigma_k}  \geq \frac{\|g_k\|_{x_k}}{\bar\kappa_M + \bar \sigma}, \ee
where $\bar\kappa_M$ is defined in \cref{lemma:tauk-bd}. Let $m_k^P := [\hat{m_k}]_{x_k} = m_k \circ R_{x_k}$ denote the pullback of the model function $m_k$. Combining \cref{eq:apphessofpullpack}, \cref{eq:tcg-stop}, and \cref{eq:xik-low-bound}, it holds
\begin{align*}  \| g_k + \Hess m^P_k(0_{x_k})[\xi_k]  \|_{x_k} & \\ & \hspace{-15ex} \leq \| (\Hess m^P_k(0_{x_k}) - \Hess m_k(x_k))[\xi_k] \|_{x_k} + \| g_k + \Hess m_k(x_k)[\xi_k] \|_{x_k} \\
& \hspace{-15ex} \leq  \beta_R \| g_k \|_{x_k} \|\xi_k\|_{x_k} + \|g_k\|_{x_k}^\theta \leq \underbracket{(\beta_R \|g_k\|_{x_k} + (\bar\kappa_M+\bar\sigma) \|g_k\|^{\theta-1}_{x_k})}_{=: \,\mathcal C_k(g_k)} \|\xi_k\|_{x_k}. \end{align*}
Similar to \cite[Proposition 5]{ring2012optimization} and applying a second order Taylor expansion, it holds
\begin{align*} {m}^P_k(\xi_k) - {m}^P_k(0_k) -  \half\iprod{g_k}{\xi_k}_{x_k} &= \half \iprod{g_k + \Hess m^P_k(\delta_k \xi_k)[\xi_k]}{\xi_k}_{x_k} \\ & \hspace{-20ex}\leq \left[ \mathcal C_k(g_k) + \|\Hess m^P_k(\delta_k \xi_k)-\Hess m^P_k(0_{x_k})\|_{x_k} \right] \|\xi_k\|^2_{x_k} = o(\|\xi_k\|_{x_k}^2), \end{align*}
where $\delta_k \in [0,1]$ is a suitable constant and we used the last estimate, $\mathcal C_k(g_k) \to 0$, and the continuity of the Hessian $\Hess m_k^P$. Therefore, due to $\rho < 0.5$ and $\alpha_0 = 1$, the full step size $\alpha_k = 1$ is chosen in \cref{eq:curvilinear} if $k$ is sufficiently large and we have $x_{k+1} = R_{x_k}(\xi_k)$. The remaining part of the proof now essentially follows \cite[Theorem 7.4.11]{opt-manifold-book} and \cite[Section 4.2.2]{Qi11}. In particular, calculating a first order Taylor expansion of the pullback gradient $\grad {\hat f}_{x_k}$ and using $\grad {\hat f}_{x_k}(0_{x_k}) = g_k$, the continuity of the pullback Hessian $\Hess {\hat f}_{x_k}$, \cref{eq:tcg-stop}, \cref{eq:apphessofpullpack}, \cref{eq:hess-mf-m}, (B.5), $\sigma_k \to 0$, and \cref{eq:bound-xikgk}, we obtain
\begin{align*} \|\grad {\hat f}_{x_k}(\xi_k)\|_{x_k} & \\ & \hspace{-13ex} \leq \|\grad {\hat f}_{x_k}(\xi_k) - g_k - \Hess {\hat f}_{x_k}(0_{x_k})[\xi_k] \|_{x_k} + \|g_k + \Hess m_k(x_k)[\xi_k]\|_{x_k} \\ & \hspace{-10ex} + \|(\Hess {\hat f}_{x_k}(0_{x_k}) - \Hess f(x_k))[\xi_k]\|_{x_k}  + \| (\Hess f(x_k) - \Hess m_k(x_k))[\xi_k]\|_{x_k} \\
& \hspace{-13ex} \leq \|\Hess {\hat f}_{x_k}(\tilde \delta_k \xi_k) - \Hess {\hat f}_{x_k}(0_{x_k})\|_{x_k} \|\xi_k\|_{x_k} + \|g_k\|^\theta_{x_k} + \beta_R \|g_k\|_{x_k} \|\xi_k\|_{x_k} \\ & \hspace{-10ex} + c \cdot \|H_k - \nabla^2 f(x_k)\| \|\xi_k\|_{x_k} + \sigma_k \|\xi_k\|_{x_k} \\ & \hspace{-13ex} = o(\|g_k\|_{x_k}), 
\end{align*} 
where $\tilde \delta_k \in [0,1]$ is again an appropriate constant. By \cite[Lemma 7.4.9]{opt-manifold-book}, this implies
\be \label{eq:th7} \frac{\|\grad f(x_{k+1})\|_{x_{k+1}}}{\|\grad f(x_{k})\|_{x_k}} \leq \tilde c \cdot \frac{\|\grad {\hat f}_{x_k}(\xi_k)\|_{x_k}}{\|g_k\|_{x_k}}\rightarrow 0, \quad \text{as} \quad k\rightarrow \infty, \ee
for some $\tilde c > 0$. Moreover, since the Hessian $\Hess f(x_*)$ is positive definite, \cite[Lemma 7.4.8]{opt-manifold-book} and \cref{eq:th7} further imply
\[ \frac{\dist(x_{k+1}, x_*)}{\dist ( x_{k}, x_*)} \rightarrow 0  \]
as $ k \to \infty$. (Here, $\dist(\cdot,\cdot)$ denotes the Riemannian geodesic distance, see \cite{opt-manifold-book}).
\end{proof}

\section{Numerical Results} \label{sec:num}
In this section, we test a variety of examples to illustrate the efficiency of
our adaptively regularized Newton method (\Optman). We mainly compare Algorithm \ref{alg:TR-ineact-newton} with the
Riemannian gradient method using the BB step size for initialization (GBB) and
the Riemannian trust region method (RTR) Manopt. All codes are written in
MATLAB. Note that Huang et al.~\cite{huang2016roptlib} implement a C-language
version of RTR to further accelerate the method. The efficiency of ARNT can also be
improved in a similar way. All experiments were performed on a workstation with Intel Xenon E5-2680 v3 processors at 2.50GHz($\times$12) and 128GB memory running CentOS 6.8 and MATLAB R2015b. %if it is implemented in the C language since .

 %and adaptive gradient method with BB initial step size, denoted by RGBB and RAdaGBB,

The default values of the GBB  parameters are set to $\rho = 10^{-4}, \delta =
0.2$,  and $\varrho = 0.85$. We have extensively tuned the stopping criterion of
the truncated CG method implemented in RTR and found that adding a rule $\|r_{j+1}\|
\leq \min\{0.1, 0.1\|r_0\|\}$ often improve the performance of RTR. All other default settings of RTR were used.
 For ARNT, we set $\eta_1 = 0.01$, $\eta_2 = 0.9$, $\gamma_0 = 0.2$, $\gamma_1 = 1$, $\gamma_2 = 10$,
 and $\sigma_k = \hat{\sigma}_k\|\grad f(x_k)\|$, where $\hat{\sigma}_k$ is
 updated by \cref{alg:tau-up} with $\hat{\sigma}_0 = 10$. The parameters in Algorithm
 \ref{alg:tN} are chosen as follows: $\rho = 10^{-4}$, $\delta = 0.2$, $\theta = 1$, and $T = 0.1$.
 Furthermore, when an estimation of the absolute value of the negative curvature, denoted by $\sigma_{est}$, is
 available at the $k$-th subproblem (see step \ref{alg:cg-S2} in Algorithm \ref{alg:tN}), we calculate
 \[ \label{alg:tau-up2} \sigma_{k+1}^{new} = \max \{\sigma_{k+1}, \,
 \sigma_{est} + \tilde \gamma \} ,\]
 with some small $\tilde\gamma \ge 0$. Then, the parameter $\sigma_{k+1}$ is
 reset to $\sigma_{k+1}^{new}$. This change does not affect our convergence results. 
 For fair comparisons, all algorithms are stopped when the norm of the Riemannian
gradient is less than $10^{-6}$ unless a different tolerance is specified. The algorithms also terminate
if a maximum number of iterations is reached. We use a maximum number of $10^4$ iteration in GBB
and $500$ in ARNT and RTR. In the implementation of ARNT and RTR, the GBB method is used to obtain a
better initial point. Here, GBB is run with stopping criterion $\|\grad f(x_k)\| \leq 10^{-3}$ and a maximum of $2000$ iterations.
The maximum number of inner iterations in ARNT is chosen adaptively depending on
the norm of the Riemannian gradient.  

In the subsequent tables, the column ``its'' represents the total
number of iterations in GBB,  while the two numbers of the column ``its'' in
ARNT and RTR are the number of outer iterations and the average numbers of inner
iterations. The columns ``f'', ``nrmG'' and ``time'' denote the final objective
value, the final norm of the  Riemannian gradient, and the CPU time that
the algorithms spent to reach the stopping criterions, respectively.

\begin{table}[t] \caption{Numerical results of Ex.~1 on low rank nearest
  correlation estimation} \label{num:ncm1}
	\centering
	\setlength{\tabcolsep}{1.1pt}
	\begin{tabular}{|c|ccc|ccc|ccc|ccc|}
		\hline
		
		& 	 \multicolumn{3}{c|}{GBB} 	& \multicolumn{3}{c|}{AdaGBB} & 	 \multicolumn{3}{c|}{ARNT}  & 	 \multicolumn{3}{c|}{RTR}\\ \hline
		$p$ 	  	 & 	 its & 	 nrmG &	 time & 	 its & 	 nrmG &	 time & 	 its & 	 nrmG &	 time & 	 its & 	 nrmG &	 time\\ \hline
        \multicolumn{13}{|c|}{$H = \mathbf{1},\, n =500$} \\ \hline
	  5 & 	 207 & 	 3.5e-7 &	 1.2 & 	 227 & 	 8.8e-7 &	 0.9 & 	 24(  14) & 	 1.2e-7 &	 1.3 & 	 31(   8) & 	 2.7e-7 &	 1.2\\ \hline 
	          10 & 	 173 & 	 8.7e-7 &	 0.5 & 	 215 & 	 9.6e-7 &	 0.5 & 	 11(  11) & 	 3.2e-7 &	 0.6 & 	 11(  12) & 	 6.7e-7 &	 0.7\\ \hline 
	          20 & 	 293 & 	 5.3e-7 &	 0.9 & 	 352 & 	 6.3e-7 &	 1.1 & 	 13(  18) & 	 1.2e-7 &	 0.9 & 	 12(  21) & 	 8.4e-7 &	 1.0\\ \hline 
	          50 & 	 2622 & 	 1.0e-6 &	 9.4 & 	 1306 & 	 8.6e-7 &	 5.8 & 	 43(  37) & 	 2.4e-7 &	 5.1 & 	 39(  20) & 	 5.5e-7 &	 3.0\\ \hline 
	         100 & 	 3286 & 	 9.0e-7 &	 17.4 & 	 3614 & 	 9.9e-7 &	 13.6 & 	 52(  51) & 	 6.0e-7 &	 11.1 & 	 52(  30) & 	 3.6e-7 &	 6.8\\ \hline 
	         150 & 	 9358 & 	 9.9e-7 &	 47.4 & 	 10000 & 	 3.4e-6 &	 62.5 & 	 51(  75) & 	 1.6e-7 &	 18.7 & 	 55(  54) & 	 5.1e-7 &	 13.8\\ \hline 
	         200 & 	 10000 & 	 2.8e-5 &	 82.1 & 	 10000 & 	 2.1e-4 &	 46.7 & 	 70(  70) & 	 5.6e-7 &	 31.0 & 	 77(  49) & 	 9.2e-7 &	 18.5\\ \hline   
   \multicolumn{13}{|c|}{$H \ne \mathbf{1},\, n=500$} \\ \hline
	5 & 	 1016 & 	 9.3e-7 &	 5.1 & 	 744 & 	 9.3e-7 &	 3.4 & 	 115(  19) & 	 1.9e-7 &	 6.0 & 	 290(  21) & 	 3.5e-7 &	 20.4\\ \hline 
	        10 & 	 722 & 	 1.0e-6 &	 3.3 & 	 431 & 	 5.6e-7 &	 1.3 & 	 40(  61) & 	 4.9e-7 &	 6.3 & 	 28(  40) & 	 6.8e-7 &	 3.6\\ \hline 
	        20 & 	 923 & 	 7.8e-7 &	 3.1 & 	 715 & 	 4.1e-7 &	 4.8 & 	 20(  70) & 	 8.2e-7 &	 4.4 & 	 23(  52) & 	 7.1e-7 &	 3.9\\ \hline 
	        50 & 	 10000 & 	 1.6e+0 &	 36.8 & 	 10000 & 	 3.1e-6 &	 65.3 & 	 69( 105) & 	 6.0e-7 &	 24.0 & 	 116( 115) & 	 7.0e-7 &	 50.3\\ \hline 
	       100 & 	 10000 & 	 1.2e-1 &	 47.0 & 	 10000 & 	 4.5e-2 &	 67.6 & 	 345( 119) & 	 5.0e-7 &	 154.8 & 	 449( 169) & 	 9.7e-7 &	 331.2\\ \hline 
	       150 & 	 10000 & 	 3.6e-1 &	 49.5 & 	 10000 & 	 5.4e-2 &	 43.8 & 	 500( 119) & 	 1.4e-1 &	 269.5 & 	 500( 168) & 	 5.9e-1 &	 385.9\\ \hline 
	       200 & 	 10000 & 	 8.3e-2 &	 65.5 & 	 10000 & 	 6.5e-2 &	 47.8 & 	 500( 125) & 	 7.0e-2 &	 341.1 & 	 500( 165) & 	 2.0e-1 &	 414.1\\ \hline 						
	\end{tabular}	
\end{table}

\subsection{Low Rank Nearest Correlation Matrix Estimation}
Given a symmetric matrix $C$ and a nonnegative symmetric weight matrix $H$, the low rank nearest correlation matrix problem is given as
\be \label{ncm}~\min_{X\in \R^{n \times n}} \, \half \|H \odot (X - C)\|_F^2, \quad\st\quad X_{ii} = 1, \quad \rank(X) \leq p, \quad X \succeq 0,  \ee
for all $i = 1,...,n$ and for $p \leq n$. By expressing $X = V^\top V$ with $V = [V_1, \ldots, V_n] \in
\R^{p\times n}$, problem \cref{ncm} can be converted into:
\[ \min_{V \in \R^{p\times n}}~\half \|H \odot(V^\top V -C)\|_F^2, \quad \st \quad \|V_i\|_2 = 1, \quad i =1,..., n. \]
In this subsection, we also use a version of the Adagrad method \cref{adagrad} in our numerical comparison. It is
dubbed as AdaGBB because its setting is similar to GBB. We select a few typical test problems as follows.

\begin{table}[t] \caption{Numerical results of Ex.~3 on low rank nearest
  correlation estimation} \label{num:ncm3}
	\centering
	\setlength{\tabcolsep}{1.0pt}
	\begin{tabular}{|c|ccc|ccc|ccc|ccc|}
		\hline
		
		& 	 \multicolumn{3}{c|}{GBB} 	& \multicolumn{3}{c|}{AdaGBB} & 	 \multicolumn{3}{c|}{ARNT}  & 	 \multicolumn{3}{c|}{RTR}\\ \hline
		$p$ 	 & 	its & 	 nrmG &	 time & 	 its & 	 nrmG &	 time & 	 its & 	 nrmG &	 time & 	 its & 	 nrmG &	 time\\ \hline
		 5 & 	 10000 & 	 1.7e+02 &	 196.5 & 	 4178 & 	 6.9e-7 &	 41.3 & 	 260(   8) & 	 9.4e-7 &	 38.1 & 	 500(  12) & 	 8.8e-2 &	 78.4\\ \hline 
		         10 & 	 10000 & 	 3.0e-4 &	 207.4 & 	 4973 & 	 8.2e-7 &	 103.8 & 	 347(  12) & 	 8.4e-7 &	 58.9 & 	 500(  17) & 	 9.3e-2 &	 102.5\\ \hline 
		         20 & 	 10000 & 	 1.5e-4 &	 198.3 & 	 5089 & 	 7.1e-7 &	 86.6 & 	 237(  24) & 	 8.3e-7 &	 63.4 & 	 500(  23) & 	 9.7e-2 &	 152.0\\ \hline 
		         50 & 	 10000 & 	 9.1e-5 &	 288.1 & 	 3675 & 	 1.0e-6 &	 90.2 & 	 34(  58) & 	 2.0e-7 &	 38.1 & 	 63(  82) & 	 7.7e-7 &	 80.2\\ \hline 
		        100 & 	 10000 & 	 3.6e-4 &	 181.6 & 	 10000 & 	 2.5e-6 &	 258.0 & 	 26( 118) & 	 7.1e-7 &	 50.1 & 	 19( 428) & 	 7.1e-7 &	 120.4\\ \hline 
		        150 & 	 10000 & 	 3.5e-2 &	 124.2 & 	 10000 & 	 4.4e-5 &	 241.7 & 	 35( 134) & 	 3.0e-7 &	 76.1 & 	 18( 688) & 	 9.0e-7 &	 173.2\\ \hline 
		        200 & 	 10000 & 	 3.5e-2 &	 153.7 & 	 10000 & 	 7.2e-5 &	 245.3 & 	 37( 130) & 	 5.5e-7 &	 78.4 & 	 16( 758) & 	 8.3e-7 &	 162.0\\ \hline  		   		
	\end{tabular}
	
\end{table}

\begin{table}[t] \caption{Numerical results of Ex.~2 on  low rank nearest
  correlation estimation} \label{num:ncm21}
	\centering
	\setlength{\tabcolsep}{1.2pt}
	\begin{tabular}{|c|ccc|ccc|ccc|ccc|}
		\hline
		
		& 	 \multicolumn{3}{c|}{GBB} 	& \multicolumn{3}{c|}{AdaGBB} & 	 \multicolumn{3}{c|}{ARNT}  & 	 \multicolumn{3}{c|}{RTR}\\ \hline
		$p$ 	 & 	 its & 	 nrmG &	 time & 	 its & 	 nrmG &	 time & 	 its & 	 nrmG &	 time & 	 its & 	 nrmG &	 time\\ \hline
		\multicolumn{13}{|c|}{$H = \mathbf{1}$ (Lymph, $n=587$)} \\ \hline
         5 & 	 252 & 	 7.3e-7 &	 1.8 & 	 335 & 	 3.2e-7 &	 1.3 & 	 19(  14) & 	 5.3e-7 &	 1.9 & 	 35(   9) & 	 3.2e-7 &	 2.3\\ \hline 
                 10 & 	 242 & 	 9.0e-7 &	 1.0 & 	 309 & 	 9.7e-7 &	 1.2 & 	 15(  16) & 	 2.7e-7 &	 1.1 & 	 28(  16) & 	 1.5e-7 &	 1.9\\ \hline 
                 20 & 	 372 & 	 9.6e-7 &	 1.9 & 	 402 & 	 9.4e-7 &	 1.8 & 	 17(  21) & 	 8.0e-7 &	 1.9 & 	 12(  21) & 	 3.6e-7 &	 1.4\\ \hline 
                 50 & 	 756 & 	 7.1e-7 &	 4.1 & 	 979 & 	 9.8e-7 &	 4.9 & 	 18(  30) & 	 5.2e-7 &	 3.0 & 	 20(  26) & 	 1.9e-7 &	 2.8\\ \hline 
                100 & 	 1614 & 	 8.9e-7 &	 10.3 & 	 2473 & 	 9.7e-7 &	 14.3 & 	 33(  54) & 	 1.9e-7 &	 10.0 & 	 28(  42) & 	 8.7e-7 &	 6.3\\ \hline 
                150 & 	 10000 & 	 1.8e-6 &	 73.2 & 	 10000 & 	 5.5e-6 &	 62.8 & 	 34(  49) & 	 5.6e-7 &	 11.2 & 	 31(  40) & 	 4.5e-7 &	 8.0\\ \hline 
                200 & 	 10000 & 	 7.0e-6 &	 79.5 & 	 10000 & 	 1.2e-5 &	 69.8 & 	 25(  44) & 	 8.9e-7 &	 8.9 & 	 28(  44) & 	 5.2e-7 &	 8.7\\ \hline   
		\multicolumn{13}{|c|}{$H \ne \mathbf{1}$ (Lymph, $n=587$)}  \\ \hline
           5 & 	 1691 & 	 9.2e-7 &	 14.1 & 	 723 & 	 6.8e-7 &	 5.2 & 	 72(  24) & 	 2.0e-7 &	 6.4 & 	 500(  18) & 	 2.8e-3 &	 43.6\\ \hline 
                  10 & 	 1774 & 	 9.8e-7 &	 13.2 & 	 742 & 	 8.4e-7 &	 4.1 & 	 34(  31) & 	 6.8e-7 &	 5.3 & 	 500(  25) & 	 9.0e-5 &	 65.0\\ \hline 
                  20 & 	 2260 & 	 8.0e-7 &	 19.5 & 	 836 & 	 1.0e-6 &	 4.8 & 	 101(  43) & 	 8.0e-7 &	 16.8 & 	 500(  36) & 	 6.9e-3 &	 93.3\\ \hline 
                  50 & 	 6468 & 	 9.7e-7 &	 47.2 & 	 2784 & 	 9.7e-7 &	 14.5 & 	 83(  83) & 	 5.5e-7 &	 29.6 & 	 114(  81) & 	 3.5e-7 &	 48.3\\ \hline 
                 100 & 	 8695 & 	 9.7e-7 &	 67.6 & 	 4897 & 	 9.9e-7 &	 26.9 & 	 26( 109) & 	 2.6e-7 &	 15.5 & 	 32( 108) & 	 2.5e-7 &	 18.3\\ \hline 
                 150 & 	 10000 & 	 2.2e-3 &	 69.8 & 	 10000 & 	 4.0e-3 &	 60.0 & 	 47( 108) & 	 7.2e-7 &	 34.4 & 	 46(  99) & 	 7.5e-7 &	 30.2\\ \hline 
                 200 & 	 10000 & 	 5.0e-3 &	 84.2 & 	 10000 & 	 2.3e-3 &	 72.9 & 	 71( 110) & 	 5.2e-7 &	 60.2 & 	 58( 129) & 	 5.9e-7 &	 57.1\\ \hline 

        \multicolumn{13}{|c|}{$H = \mathbf{1}$ (ER, $n =692$)} \\ \hline
		 5 & 	 359 & 	 9.8e-7 &	 2.7 & 	 500 & 	 5.7e-7 &	 2.3 & 	 20(  14) & 	 1.9e-7 &	 1.8 & 	 49(  13) & 	 5.9e-7 &	 3.4\\ \hline 
		         10 & 	 198 & 	 7.8e-7 &	 0.9 & 	 330 & 	 9.9e-7 &	 1.8 & 	 12(  20) & 	 2.1e-7 &	 1.6 & 	 14(  16) & 	 9.8e-8 &	 1.5\\ \hline 
		         20 & 	 384 & 	 5.2e-7 &	 2.0 & 	 481 & 	 9.9e-7 &	 3.1 & 	 20(  31) & 	 4.6e-7 &	 3.5 & 	 30(  25) & 	 1.3e-7 &	 4.2\\ \hline 
		         50 & 	 585 & 	 6.8e-7 &	 3.5 & 	 951 & 	 9.3e-7 &	 7.3 & 	 17(  34) & 	 1.5e-7 &	 4.5 & 	 16(  29) & 	 2.9e-7 &	 3.7\\ \hline 
		        100 & 	 6176 & 	 1.0e-6 &	 43.6 & 	 5882 & 	 9.7e-7 &	 34.7 & 	 54(  70) & 	 6.9e-7 &	 30.3 & 	 47(  66) & 	 1.2e-7 &	 22.1\\ \hline 
		        150 & 	 1198 & 	 9.7e-7 &	 10.1 & 	 2895 & 	 9.5e-7 &	 21.0 & 	 26(  58) & 	 6.8e-7 &	 14.4 & 	 20(  51) & 	 5.1e-7 &	 8.8\\ \hline 
		        200 & 	 10000 & 	 6.7e-6 &	 91.7 & 	 10000 & 	 4.6e-5 &	 79.2 & 	 60(  68) & 	 6.1e-7 &	 40.8 & 	 29(  57) & 	 7.1e-7 &	 15.6\\ \hline 
		\multicolumn{13}{|c|}{$H \ne \mathbf{1}$ (ER, $n =692$)} \\ \hline
		  5 & 	 1382 & 	 9.2e-7 &	 13.7 & 	 708 & 	 9.6e-7 &	 3.8 & 	 24(  23) & 	 8.0e-7 &	 4.2 & 	 500(  20) & 	 2.8e-2 &	 48.6\\ \hline 
		         10 & 	 1686 & 	 9.3e-7 &	 15.8 & 	 640 & 	 8.1e-7 &	 3.5 & 	 27(  33) & 	 3.3e-7 &	 5.6 & 	 500(  26) & 	 9.7e-4 &	 63.5\\ \hline 
		         20 & 	 2123 & 	 9.6e-7 &	 18.5 & 	 1159 & 	 9.8e-7 &	 6.6 & 	 138(  39) & 	 6.8e-7 &	 32.3 & 	 500(  41) & 	 5.4e-4 &	 107.8\\ \hline 
		         50 & 	 4923 & 	 6.3e-7 &	 31.4 & 	 3138 & 	 1.0e-6 &	 31.3 & 	 100(  79) & 	 9.4e-7 &	 48.4 & 	 214(  83) & 	 8.1e-7 &	 103.3\\ \hline 
		        100 & 	 10000 & 	 1.1e-4 &	 71.6 & 	 10000 & 	 1.2e-6 &	 65.7 & 	 49( 110) & 	 1.6e-7 &	 43.6 & 	 56( 163) & 	 1.7e-7 &	 64.8\\ \hline 
		        150 & 	 10000 & 	 4.4e-4 &	 84.0 & 	 5508 & 	 9.5e-7 &	 39.2 & 	 33( 117) & 	 4.9e-7 &	 35.9 & 	 35( 145) & 	 8.4e-7 &	 42.2\\ \hline 
		        200 & 	 10000 & 	 2.1e-3 &	 94.0 & 	 10000 & 	 4.7e-4 &	 79.0 & 	 83( 159) & 	 8.6e-7 &	 132.9 & 	 52( 221) & 	 3.6e-7 &	 102.6\\ \hline  
		
	\end{tabular}
	
\end{table}

\begin{enumerate}
  \item[Ex. 1.] Let $n = 500$ and let $C_{ij} = 0.5 + e^{-0.05|i-j|}$
    for $i,j = 1,..., n$. The weight matrix $H$ is either $\mathbf{1}$ or a
    random matrix whose entries are mostly uniformly distributed in $[0.1,10]$
    except that $200$ entries are distributed in $[0.01,100]$.
  \item[Ex. 2.] The matrix $C$ is obtained from the real gene correlation
    matrices such as Lymph, ER, Hereditarybc and Leukemia. The weight matrix $H$
    is either $\mathbf{1}$ or a random matrix whose entries are set
    as in Ex. 1.% are uniformly distributed in $[0.1,10]$ except for $200$ entries in $[0.01,100]$.
  \item[Ex. 3.] Let $n = 943$. The matrix $C$ is based on $100,000$ ratings for
    $1682$ movies by $943$ users from
    the Movielens data sets. The weight matrix $H$ is provided by T. Fushiki at Institute of Statistical Mathematics, Japan.
\end{enumerate}

The detailed numerical results are reported in Tables
\ref{num:ncm1}-\ref{num:ncm22}. For Ex.~1, all methods perform well if $p$ is
small. For the cases with $H\neq \mathbf{1}$, ARNT is the best when $p=50$ and
$p=100$ while all of them fail 
  when $p=150$ and $p=200$.
 For Ex.~2,  GBB may not converge when $p$ is large, and ARNT is efficient
whenever $p$ is small or large. In
particular, ARNT is better than RTR on ER and Leukemia with $H\neq \mathbf{1}$
and RTR may fail on a few instances. 
 %and RTR spends much more time when $p$ is small.
For  Ex.~3,  we can see that GBB and
RTR fail to converge when $p$ is small, while ARNT and AdaGBB still work. In fact, we
observe negative curvatures of the Hessian at many iterations of  
ARNT and the strategy \cref{eq:subline} indeed helps the convergence.

\begin{table} \caption{Numerical results of Ex.~2 on low rank nearest
  correlation estimation (continued)} \label{num:ncm22}
	\centering
	\setlength{\tabcolsep}{1.5pt}
	\begin{tabular}{|c|ccc|ccc|ccc|ccc|}
		\hline
		
		& 	 \multicolumn{3}{c|}{GBB} 	& \multicolumn{3}{c|}{AdaGBB} & 	 \multicolumn{3}{c|}{ARNT}  & 	 \multicolumn{3}{c|}{RTR}\\ \hline
		$p$ 	 & 	 its & 	 nrmG &	 time & 	 its & 	 nrmG &	 time & 	 its & 	 nrmG &	 time & 	 its & 	 nrmG &	 time\\ \hline
		\multicolumn{13}{|c|}{$H = \mathbf{1}$ (Hereditarybc, $n =1869$)} \\ \hline
         5 & 	 156 & 	 7.9e-7 &	 4.6 & 	 189 & 	 4.0e-8 &	 3.9 & 	 11(  14) & 	 9.8e-7 &	 8.9 & 	 13(   7) & 	 1.5e-7 &	 6.2\\ \hline 
                 10 & 	 157 & 	 8.8e-7 &	 3.5 & 	 220 & 	 9.3e-7 &	 6.5 & 	 3(  10) & 	 9.8e-8 &	 5.6 & 	 3(  10) & 	 9.7e-8 &	 5.6\\ \hline 
                 20 & 	 299 & 	 6.3e-7 &	 7.0 & 	 304 & 	 1.9e-7 &	 9.6 & 	 17(  23) & 	 4.5e-7 &	 18.6 & 	 18(  25) & 	 4.7e-7 &	 16.8\\ \hline 
                 50 & 	 10000 & 	 9.7e-5 &	 254.3 & 	 10000 & 	 7.5e-5 &	 266.5 & 	 32(  15) & 	 8.5e-7 &	 23.2 & 	 33(  18) & 	 8.8e-7 &	 23.5\\ \hline 
                100 & 	 10000 & 	 1.7e-5 &	 294.3 & 	 10000 & 	 6.6e-5 &	 369.1 & 	 33(  14) & 	 5.2e-7 &	 24.8 & 	 33(  17) & 	 4.5e-7 &	 25.9\\ \hline 
                150 & 	 10000 & 	 5.3e-5 &	 345.3 & 	 10000 & 	 1.1e-4 &	 352.4 & 	 34(  14) & 	 7.5e-7 &	 27.1 & 	 34(  16) & 	 6.4e-7 &	 29.5\\ \hline 
                200 & 	 10000 & 	 2.7e-5 &	 372.3 & 	 10000 & 	 3.7e-5 &	 342.1 & 	 35(  15) & 	 5.9e-7 &	 34.1 & 	 36(  20) & 	 6.7e-7 &	 35.2\\ \hline    
		\multicolumn{13}{|c|}{$H \ne \mathbf{1}$ (Hereditarybc, $n =1869$)} \\ \hline
         5 & 	 256 & 	 6.6e-7 &	 6.2 & 	 238 & 	 9.3e-7 &	 6.3 & 	 12(  15) & 	 6.6e-7 &	 9.9 & 	 52(  15) & 	 4.8e-7 &	 28.5\\ \hline 
                 10 & 	 196 & 	 7.8e-7 &	 5.3 & 	 242 & 	 9.6e-7 &	 7.3 & 	 8(  17) & 	 7.1e-7 &	 7.8 & 	 7(  17) & 	 5.6e-7 &	 7.2\\ \hline 
                 20 & 	 361 & 	 9.9e-7 &	 12.0 & 	 315 & 	 9.6e-7 &	 10.6 & 	 7(  19) & 	 1.2e-7 &	 11.7 & 	 6(  18) & 	 6.8e-7 &	 10.8\\ \hline 
                 50 & 	 10000 & 	 1.2e-3 &	 303.5 & 	 10000 & 	 1.9e-3 &	 317.7 & 	 39(  23) & 	 9.9e-8 &	 38.8 & 	 38(  30) & 	 3.7e-7 &	 43.4\\ \hline 
                100 & 	 10000 & 	 1.3e-3 &	 352.9 & 	 10000 & 	 2.2e-3 &	 338.9 & 	 34(  22) & 	 8.4e-7 &	 38.0 & 	 39(  29) & 	 5.3e-7 &	 49.2\\ \hline 
                150 & 	 10000 & 	 2.4e-3 &	 386.3 & 	 10000 & 	 4.1e-3 &	 389.9 & 	 38(  24) & 	 5.6e-7 &	 45.6 & 	 41(  31) & 	 4.9e-7 &	 60.0\\ \hline 
                200 & 	 10000 & 	 1.8e-3 &	 410.0 & 	 10000 & 	 7.3e-4 &	 359.1 & 	 35(  24) & 	 7.7e-7 &	 52.6 & 	 41(  30) & 	 7.3e-7 &	 63.1\\ \hline 
        \multicolumn{13}{|c|}{$H = \mathbf{1}$ (Leukemia, $n =1255$)} \\ \hline
		  5 & 	 272 & 	 8.9e-7 &	 4.2 & 	 261 & 	 4.8e-7 &	 2.9 & 	 15(  16) & 	 4.0e-7 &	 5.6 & 	 23(   9) & 	 4.0e-7 &	 5.2\\ \hline 
		          10 & 	 540 & 	 9.6e-7 &	 12.5 & 	 453 & 	 8.2e-7 &	 5.7 & 	 23(  20) & 	 5.6e-7 &	 8.4 & 	 48(  21) & 	 6.8e-7 &	 13.8\\ \hline 
		          20 & 	 1064 & 	 9.6e-7 &	 23.0 & 	 1602 & 	 1.0e-6 &	 26.2 & 	 34(  31) & 	 4.1e-7 &	 14.8 & 	 131(  25) & 	 1.5e-7 &	 39.4\\ \hline 
		          50 & 	 1917 & 	 8.5e-7 &	 32.9 & 	 2535 & 	 4.0e-7 &	 35.8 & 	 33(  49) & 	 2.6e-7 &	 26.0 & 	 28(  42) & 	 1.6e-7 &	 17.4\\ \hline 
		         100 & 	 10000 & 	 4.1e-5 &	 169.4 & 	 10000 & 	 2.4e-5 &	 156.9 & 	 35(  27) & 	 9.6e-7 &	 19.5 & 	 35(  28) & 	 5.6e-7 &	 18.0\\ \hline 
		         150 & 	 10000 & 	 3.1e-5 &	 194.8 & 	 10000 & 	 9.2e-5 &	 184.4 & 	 40(  25) & 	 4.8e-7 &	 24.4 & 	 36(  27) & 	 9.5e-7 &	 20.0\\ \hline 
		         200 & 	 10000 & 	 3.2e-5 &	 232.1 & 	 10000 & 	 3.9e-4 &	 200.1 & 	 37(  25) & 	 5.0e-7 &	 24.8 & 	 36(  27) & 	 4.6e-7 &	 22.5\\ \hline   
       \multicolumn{13}{|c|}{$H \ne \mathbf{1}$ (Leukemia, $n=1255$)} \\ \hline
         5 & 	 1404 & 	 5.9e-7 &	 55.0 & 	 762 & 	 5.5e-7 &	 13.3 & 	 44(  20) & 	 3.8e-7 &	 16.3 & 	 500(  16) & 	 3.7e-3 &	 137.7\\ \hline 
                 10 & 	 680 & 	 9.7e-7 &	 21.8 & 	 608 & 	 9.8e-7 &	 13.0 & 	 23(  22) & 	 9.5e-7 &	 11.2 & 	 500(  20) & 	 2.1e-3 &	 169.5\\ \hline 
                 20 & 	 2461 & 	 9.4e-7 &	 77.1 & 	 2250 & 	 9.3e-7 &	 51.6 & 	 59(  32) & 	 9.5e-7 &	 31.0 & 	 500(  34) & 	 3.5e-4 &	 289.0\\ \hline 
                 50 & 	 3354 & 	 9.7e-7 &	 82.0 & 	 1790 & 	 7.8e-7 &	 47.1 & 	 33(  86) & 	 4.8e-7 &	 48.2 & 	 58(  74) & 	 1.2e-7 &	 79.2\\ \hline 
                100 & 	 10000 & 	 1.8e-2 &	 170.5 & 	 10000 & 	 1.8e-3 &	 158.6 & 	 36(  51) & 	 7.7e-7 &	 37.5 & 	 44(  53) & 	 4.9e-7 &	 48.7\\ \hline 
                150 & 	 10000 & 	 3.4e-3 &	 194.9 & 	 10000 & 	 2.1e-3 &	 197.8 & 	 43(  52) & 	 4.8e-7 &	 48.2 & 	 51(  52) & 	 4.8e-7 &	 57.1\\ \hline 
                200 & 	 10000 & 	 3.9e-3 &	 216.5 & 	 10000 & 	 4.3e-2 &	 205.8 & 	 46(  50) & 	 4.6e-7 &	 55.9 & 	 50(  52) & 	 6.5e-7 &	 65.4\\ \hline    
		
	\end{tabular}
	
\end{table}

\subsection{Simple Nonlinear Eigenvalue Problems}% \cite{zhao2015riemannian}}
A simplified model problem for density functional theory is given by
\[ \min_{X \in \R^{n \times k}}~\frac{1}{2} \tr(X^\top LX) + \frac{\alpha}{4}\rho(X)^\top L^\dag(\rho(X)), \quad \st \quad X^\top X = I_k, \]
where $L \in \R^{n \times n}$ is a symmetric matrix and $\rho(X)$ is a vector
whose components are the diagonal elements of $XX^\top$.

In this numerical experiment, $L$ is set to a tridiagonal matrix whose main
diagonal elements are $2$ and the secondary diagonal elements are $-1$. A series of
experiments using different values of $n,k$ and $\alpha$ are conducted.
Specifically, in the first case, we fix $p = 50,\alpha =1$, and try different $n$
ranging from $2000$ to $50000$. Then, we set $n =10000, \alpha = 1$ and vary $p$
from $20$ to $100$. At last, the performance on different values for $\alpha$ is also compared.
The detailed numerical results are reported in Tables \ref{num:noneig1},
\ref{num:noneig2} and \ref{num:noneig3}, respectively. In Table
\ref{num:noneig1},  we can see that ARNT is most efficient, while the
performance of RTR sometimes is not stable. Similar results are shown in Table
\ref{num:noneig2}. In Table \ref{num:noneig3}, ARNT is better than RTR,
especially for large $\alpha$. We also observe that GBB often performs comparable to ARNT in CPU time in this example.

\subsection{Kohn-Sham Total Energy Minimization} %\cite{wen2013adaptive}
Using a suitable discretization scheme, we can formulate a finite dimensional
approximation to the continuous KS minimization problem \cite{wen2013adaptive}  as
\[ \min_{X \in \C^{n \times p}}~f(X) \quad \st \quad X^*X = I,  \]
where $f(X):= \frac{1}{4} \tr(X^*LX) + \half \tr(X^*V_{ion}X) + \half
\sum_i\sum_l|x_i^*\omega_l| + \frac{1}{4}\rho L^\dag + \half e^\top
\epsilon_{xc} (\rho)$, $X = [x_1, \cdots, x_p] \in \C^{n \times p}$,  $\rho(X)
:= \diag(XX^*)$, $L$ is a finite dimensional Laplacian operator, $V_{ion}$
corresponds to the ionic pseudopotentials, $w_l$ represents a discretized
pseudopotential reference projection function and  $\epsilon_{xc}$ is related to
the exchange correlation energy.%. The matrix $L^\dag$ corresponds to the pseudoinverse of $L$ and the fourth term denotes the Hartree potential energy, which is used to model the classical electrostatic average interaction between electrons. The final term denotes the exchange correlation energy, which is used to describe the nonclassical interaction between electrons.

Our experiments are based on the KSSOLV package \cite{YangMezaLeeWang2009}. As in \cite{wen2013adaptive}, we use the Wirtinger calculus \cite{kreutz2009complex} to compute the complex gradient and Hessian of the function $f$. Let us also note that the Lipschitz continuity required in Assumption (A.1) may not be satisfied for all types of exchange correlations. However, for the correlation that is defined by the Perdew-Zunger formula and used in this example, Lipschitz continuity was established in \cite[Lemma 3.3]{ulbrich2015proximal}. In
addition to GBB and  RTR, we further compare ARNT with the self-consistent field
(SCF) iteration and the regularized trust-region method TRQH in
\cite{wen2013adaptive}. In the implementation of TRQH, RTR and ARNT, we use the
same initial point obtained by GBB. %until the norm of the Riemannian gradient is
%less than $10^{-3}$ or 2000 iterations are reached. 
Note that TRQH
essentially coincides with ARNT except that the subproblem \cref{eq:tr-sub} is
solved by GBB.

\begin{table}[t] \caption{Numerical results on nonlinear eigenspace with fixed $(p,\alpha) = (30,1000)$.} \label{num:noneig1}
\centering
\setlength{\tabcolsep}{1.5pt}
\begin{tabular}{|c|ccc|ccc|ccc|}
  \hline
& 	 \multicolumn{3}{c|}{GBB} & 	 \multicolumn{3}{c|}{ARNT}  & 	 \multicolumn{3}{c|}{RTR}\\ \hline
$n$	 & 	 its & 	 nrmG &	 time & 	 its & 	 nrmG &	 time & 	 its & 	 nrmG &	 time  \\ \hline
  2000 & 	 1204 & 	 1.0e-6 &	 4.9 & 	 52(21) & 	 4.1e-7 &	 6.8 & 	 251(24) & 	 7.9e-7 &	 19.0\\ \hline 
  3000 & 	 1648 & 	 1.2e-7 &	 11.0 & 	 19(18) & 	 8.5e-7 &	 4.1 & 	 60(23) & 	 8.8e-7 &	 7.3\\ \hline 
  5000 & 	 1111 & 	 7.0e-7 &	 11.6 & 	 38(22) & 	 4.9e-7 &	 10.7 & 	 141(24) & 	 7.8e-7 &	 23.2\\ \hline 
  8000 & 	 1389 & 	 9.9e-7 &	 17.7 & 	 33(19) & 	 4.4e-7 &	 14.4 & 	 132(24) & 	 8.9e-7 &	 34.2\\ \hline 
  10000 & 	 1757 & 	 1.0e-6 &	 41.4 & 	 48(23) & 	 7.6e-7 &	 24.4 & 	 57(26) & 	 7.9e-7 &	 25.1\\ \hline 

\end{tabular}
\end{table}

\begin{table}[t] \caption{Numerical results on nonlinear eigenspace with fixed $(n, \alpha) = (5000,1000)$.} \label{num:noneig2}
\centering
\setlength{\tabcolsep}{1.5pt}
\begin{tabular}{|c|ccc|ccc|ccc|}
  \hline
& 	 \multicolumn{3}{c|}{GBB} & 	 \multicolumn{3}{c|}{ARNT}  & 	 \multicolumn{3}{c|}{RTR}\\ \hline
  $ p $ 	 & 	 its & 	 nrmG &	 time & 	 its & 	 nrmG &	 time & 	 its & 	 nrmG &	 time\\ \hline
 10 & 	 341 & 	 5.9e-7 &	 1.7 & 	 4(10) & 	 4.5e-7 &	 0.4 & 	 3(10) & 	 4.4e-7 &	 0.3\\ \hline 
 20 & 	 610 & 	 8.0e-7 &	 3.2 & 	 5(15) & 	 3.5e-7 &	 1.9 & 	 3(16) & 	 3.3e-7 &	 1.7\\ \hline 
 30 & 	 1111 & 	 7.0e-7 &	 9.4 & 	 38(22) & 	 4.9e-7 &	 8.7 & 	 141(24) & 	 7.8e-7 &	 19.6\\ \hline 
 50 & 	 3627 & 	 9.4e-7 &	 62.1 & 	 46(26) & 	 7.6e-7 &	 31.9 & 	 500(37) & 	 1.5e-3 &	 175.9\\ \hline  
\end{tabular}
\end{table}

\begin{table}[t] \caption{Numerical results on nonlinear eigenspace with fixed $(n, p) = (8000,30)$.} \label{num:noneig3}
\centering
\setlength{\tabcolsep}{1.5pt}
\begin{tabular}{|c|ccc|ccc|ccc|}
  \hline
& 	 \multicolumn{3}{c|}{GBB} & 	 \multicolumn{3}{c|}{ARNT}  & 	 \multicolumn{3}{c|}{RTR}\\ \hline
  $ \alpha $ 	 & 	 its & 	 nrmG &	 time & 	 its & 	 nrmG &	 time & 	 its & 	 nrmG &	 time\\ \hline
  1 & 	 194 & 	 6.8e-7 &	 2.4 & 	 3(28) & 	 8.3e-7 &	 1.9 & 	 3(28) & 	 6.0e-7 &	 1.7\\ \hline 
  10 & 	 299 & 	 3.7e-7 &	 3.7 & 	 3(36) & 	 4.5e-7 &	 2.3 & 	 3(36) & 	 3.6e-7 &	 2.0\\ \hline 
  100 & 	 572 & 	 9.1e-7 &	 7.3 & 	 3(26) & 	 5.2e-7 &	 5.0 & 	 3(26) & 	 5.2e-7 &	 4.8\\ \hline 
  1000 & 	 1389 & 	 9.9e-7 &	 18.2 & 	 33(19) & 	 4.4e-7 &	 13.1 & 	 132(24) & 	 8.9e-7 &	 33.1\\ \hline  
\end{tabular}
\end{table}

\begin{table} \caption{Numerical results on KS total energy minimization.} \label{tab:ks}
	\centering
    \setlength{\tabcolsep}{1.5pt}
	\begin{tabular}{|c|cccc||cccc|}
		\hline		
solver  & f & its  & nrmG & time & f  & its  & nrmG & time	\\ \hline

& 	 \multicolumn{4}{c||}{alanine} & \multicolumn{4}{c|}{al} \\ \hline 
SCF   & -6.1162e+01 &   14 &  3.9e-7  &     25.0 	 & -1.5784e+01 &  101 &  4.5e-2  &    146.9 \\ \hline 
 OptM  & -6.1162e+01 &   80 &  7.1e-7 &     25.5 	&  -1.5804e+01 & 1461 &  9.9e-7 &    391.1  \\ \hline 
TRQH   & -6.1162e+01 &    6(  16) &  6.5e-7 &     39.7 	 & -1.5804e+01 &   39(  16) &  9.6e-7 &    411.9 \\ \hline 
ARNT & -6.1162e+01 &    3(   9) & 	 3.9e-7 &	     24.4 	& -1.5804e+01 &    5( 113) & 	 3.5e-7 &	    196.4 \\ \hline 
RTR & -6.1162e+01 &    3(   9) & 	 4.1e-7 &	     24.3 	& -1.5804e+01 &    5( 108) & 	 9.9e-8 &	    188.9 \\ \hline 
& 	 \multicolumn{4}{c||}{benzene} & \multicolumn{4}{c|}{c12h26} \\ \hline 
SCF   & -3.7226e+01 &   13 &  4.0e-7  &     14.3 	 & -8.1536e+01 &   13 &  9.1e-7  &     30.2 \\ \hline 
 OptM  & -3.7226e+01 &   68 &  5.1e-7 &     13.4 	&  -8.1536e+01 &   89 &  8.8e-7 &     34.1  \\ \hline 
TRQH   & -3.7226e+01 &    6(  12) &  9.3e-7 &     19.2 	 & -8.1536e+01 &    7(  12) &  9.7e-7 &     50.0 \\ \hline 
ARNT & -3.7226e+01 &    3(  10) & 	 9.2e-8 &	     13.3 	& -8.1536e+01 &    3(  13) & 	 6.4e-7 &	     29.5 \\ \hline 
RTR & -3.7226e+01 &    3(  10) & 	 8.1e-8 &	     13.6 	& -8.1536e+01 &    3(  13) & 	 5.2e-7 &	     29.5 \\ \hline 
& 	 \multicolumn{4}{c||}{c2h6} & \multicolumn{4}{c|}{co2} \\ \hline 
SCF   & -1.4420e+01 &   10 &  6.8e-7  &      2.5 	 & -3.5124e+01 &   10 &  3.1e-7  &      2.6 \\ \hline 
 OptM  & -1.4420e+01 &   59 &  9.1e-7 &      2.6 	&  -3.5124e+01 &   59 &  5.2e-7 &      2.6  \\ \hline 
TRQH   & -1.4420e+01 &    6(  12) &  8.7e-7 &      4.0 	 & -3.5124e+01 &    6(  12) &  3.7e-7 &      3.9 \\ \hline 
ARNT & -1.4420e+01 &    3(   8) & 	 4.7e-7 &	      2.5 	& -3.5124e+01 &    3(   9) & 	 3.1e-7 &	      2.5 \\ \hline 
RTR & -1.4420e+01 &    3(   7) & 	 3.9e-7 &	      2.7 	& -3.5124e+01 &    3(  10) & 	 2.5e-7 &	      2.7 \\ \hline 
& 	 \multicolumn{4}{c||}{ctube661} & \multicolumn{4}{c|}{graphene16} \\ \hline 
SCF   & -1.3464e+02 &   16 &  3.1e-7  &     88.5 	 & -9.4028e+01 &  101 &  5.8e-4  &    160.0 \\ \hline 
 OptM  & -1.3464e+02 &  101 &  7.2e-7 &     93.0 	&  -9.4046e+01 &  187 &  8.5e-7 &     40.8  \\ \hline 
TRQH   & -1.3464e+02 &    6(  19) &  3.2e-7 &    138.5 	 & -9.4046e+01 &    8(  19) &  9.5e-7 &     70.3 \\ \hline 
ARNT & -1.3464e+02 &    3(  11) & 	 4.9e-7 &	     78.3 	& -9.4046e+01 &    3(  19) & 	 8.6e-7 &	     40.3 \\ \hline 
RTR & -1.3464e+02 &    3(  11) & 	 4.2e-7 &	     78.2 	& -9.4046e+01 &    3(  19) & 	 7.3e-7 &	     40.7 \\ \hline 
& 	 \multicolumn{4}{c||}{graphene30} & \multicolumn{4}{c|}{h2o} \\ \hline 
SCF   & -1.7358e+02 &  101 &  2.2e-3  &    860.6 	 & -1.6441e+01 &    9 &  1.4e-7  &      1.8 \\ \hline 
 OptM  & -1.7360e+02 &  378 &  6.5e-7 &    517.0 	&  -1.6441e+01 &   58 &  8.9e-7 &      2.0  \\ \hline 
TRQH   & -1.7360e+02 &   12(  38) &  8.6e-7 &    783.9 	 & -1.6441e+01 &    5(  38) &  8.4e-7 &      2.9 \\ \hline 
ARNT & -1.7360e+02 &    4(  33) & 	 2.5e-7 &	    446.8 	& -1.6441e+01 &    3(  11) & 	 3.9e-7 &	      1.8 \\ \hline 
RTR & -1.7360e+02 &  100(   4) & 	 2.3e-5 &	    828.8 	& -1.6441e+01 &    3(  11) & 	 3.1e-7 &	      2.1 \\ \hline 
& 	 \multicolumn{4}{c||}{hnco} & \multicolumn{4}{c|}{nic} \\ \hline 
SCF   & -2.8635e+01 &   12 &  3.5e-7  &      3.3 	 & -2.3544e+01 &   10 &  7.2e-7  &      1.2 \\ \hline 
 OptM  & -2.8635e+01 &  131 &  9.7e-7 &      5.6 	&  -2.3544e+01 &   63 &  9.9e-7 &      1.1  \\ \hline 
TRQH   & -2.8635e+01 &    7(  21) &  9.5e-7 &      6.9 	 & -2.3544e+01 &    8(  21) &  9.3e-7 &      2.3 \\ \hline 
ARNT & -2.8635e+01 &    3(  15) & 	 7.5e-7 &	      3.7 	& -2.3544e+01 &    3(   8) & 	 4.4e-7 &	      1.0 \\ \hline 
RTR & -2.8635e+01 &    3(  16) & 	 7.7e-7 &	      4.5 	& -2.3544e+01 &    3(   8) & 	 4.6e-7 &	      1.3 \\ \hline 
& 	 \multicolumn{4}{c||}{ptnio} & \multicolumn{4}{c|}{qdot} \\ \hline 
SCF   & -2.2679e+02 &   66 &  7.7e-7  &    146.2 	 & 2.7702e+01 &  101 &  3.4e-2  &     22.3 \\ \hline 
 OptM  & -2.2679e+02 &  495 &  5.3e-7 &    145.6 	&  2.7695e+01 & 2000 &  3.3e-6 &     70.8  \\ \hline 
TRQH   & -2.2679e+02 &   23(  39) &  9.3e-7 &    286.0 	 & 2.7695e+01 &   91(  39) &  9.9e-7 &    115.8 \\ \hline 
ARNT & -2.2679e+02 &    4(  52) & 	 6.9e-7 &	    132.4 	& 2.7695e+01 &   27(  65) & 	 7.1e-7 &	     64.5 \\ \hline 
RTR & -2.2679e+02 &    4(  46) & 	 8.5e-7 &	    122.5 	& 2.7695e+01 &   37(  68) & 	 4.0e-7 &	     83.3 \\ \hline 
%& 	 \multicolumn{4}{c||}{si2h4} & \multicolumn{4}{c|}{sih4} \\ \hline 
%SCF   & -6.3010e+00 &   11 &  8.6e-7  &      2.6 	 & -6.1769e+00 &    8 &  1.9e-7  &      1.4 \\ \hline 
% OptM  & -6.3010e+00 &   85 &  9.6e-7 &      3.4 	&  -6.1769e+00 &   45 &  5.4e-7 &      1.6  \\ \hline 
%TRQH   & -6.3010e+00 &    6(  16) &  9.5e-7 &      4.7 	 & -6.1769e+00 &    5(  16) &  9.7e-7 &      2.5 \\ \hline 
%ARNT & -6.3010e+00 &    3(  10) & 	 3.6e-7 &	      2.9 	& -6.1769e+00 &    3(   7) & 	 5.7e-7 &	      1.6 \\ \hline 
%RTR & -6.3010e+00 &    3(  10) & 	 2.6e-7 &	      3.2 	& -6.1769e+00 &    3(   7) & 	 5.8e-7 &	      1.9 \\ \hline 

	\end{tabular}
	
\end{table}

A summary of the computational results is given in Table \ref{tab:ks}. All
algorithms reach the same objective function value when the gradient norm
criterion is satisfied. ARNT and RTR take a small number of outer iterations to
converge and often exhibit a fast convergence rate. In particular, ARNT tends to be
more efficient than other algorithms on ``graphene30'' and ``qdot''. It can be even
faster than SCF when SCF works well. ARNT also outperforms TRQH. This shows
that the accuracy of solving the subproblem \cref{eq:tr-sub} is indeed
important. 

\subsection{Bose-Einstein Condensates (BEC)} %\cite{wu2015regularized}
The total energy in BEC is defined as
\[ E(\psi) = \int_{\R^d} \left[ \half |\nabla \psi(\xb)|^2 + V(\xb)|\psi(\xb)|^2 + \frac{\beta}{2}|\psi(\xb)|^4 - \Omega \bar{\psi}(\xb)L_z(\xb)\right] d\xb,  \]
where $\xb \in \R^d$ is the spatial coordinate vector, $\bar{\psi}$ denotes the
complex conjugate of $\psi$, $L_z = -i(x\partial - y\partial x),\, V(x)$ is an external trapping potential, and $\beta, \Omega$ are given constants.
%Then the ground state of a BEC problem is defined as the following optimization problem:
%\be \label{prob:bec1} \psi_g = \arg \min_{\psi \in S} \quad E(\psi),\ee
%where the spherical constraint $S$ is defined as
%\be \label{prob:bec2} S = \left \{ \psi \mid E(\psi) < \infty, \, \int_{\R^d} |\psi(\xb)|^2 dx = 1 \right\}.\ee
Using a suitable discretization, e.g., such as finite differences or the sine
pseudospectral and Fourier pseudospectral (FP) method, we can reformulate the BEC problem as follows
\[ \min_{x \in \C^M} ~ f(x) := \half x^*Ax + {\frac{\beta}{2}}\sum_{j =1}^M |x_j|^4, \quad \st \quad \|x\|_2 = 1, \]
where $M \in \N$, {$\beta$ is a given real constant}, and $A \in \C^{M\times M}$ is a Hermitian matrix.

In this numerical experiment, we again use the Wirtinger calculus to calculate the complex gradient and Hessian of the objective function. We stop GBB, ARNT, RTR, and TRQH (the Newton method
in \cite{wu2015regularized}) when the gradient norm
is less than $10^{-4}$ or the maximum number of iterations is reached. For TRQH,
the stopping criterion 
\[ \|x^{k+1} - x^{k}\|_\infty \leq \epsilon_x \]
is added for some small constant $\epsilon_x$ since TRQH often does not converge
under the gradient norm criterion. We take $d = 2$ and test two different
potential functions
\[ V_1(x,y) = \half x^2 + \half y^2, \quad \text{and} \quad V_2(x,y) = -0.1(x^2 + y^2) +0.3((x^2 + y^2)/2)^2. \]
The BEC problem is discretized by FP on the bounded domain $(-16,16)^2$ with $\beta =
500, 1000$ and different values of $\Omega$ ranging from $0$ to $0.95$. Under
the same settings as in \cite[Section 4.3]{wu2015regularized}, we use the mesh
refinement procedure with the coarse meshes  $(2^4+1) \times (2^4+1), (2^5+1) \times
(2^5+1), \ldots, (2^7+1) \times (2^7+1)$  to gradually obtain an initial solution
point on the finest mesh $(2^8+1) \times (2^8+1)$. For a fair
comparison, all algorithms are tested with mesh refinement and start from
the same initial point on the coarsest mesh with 
$\phi(x,y) = \frac{(1-\Omega)\phi_1(x,y) + \Omega \phi_2(x,y)}{\|(1-\Omega)\phi_1(x,y) + \Omega \phi_2(x,y)\|}$ and $\phi_1(x,y) = \frac{1}{\sqrt{\pi}} e^{-(x^2 +y^2)/2},\, \phi_2(x,y) = \frac{x+iy}{\sqrt{\pi}}e^{-(x^2+y^2)}/2$.

A summary of the results is presented in the Tables \ref{tab:bec1}--\ref{tab:bec2} for the potential functions $V_1$ and $V_2$, respectively. The
parameter $\epsilon_x$ for TRQH is set to  $10^{-8}$ and $10^{-7}$ in these two
cases. The tables show that GBB does not to converge
within $10000$ steps in several cases. TRQH usually performs worse than ARNT in terms of accuracy and time
except in the cases $\beta = 1000$ with $\Omega = 0.95$ in Table \ref{tab:bec1}
where ARNT finds a point with a smaller objective function value. ARNT
performs not
worse than  RTR in most experiments. %In Table \ref{tab:bec2}, we set $\epsilon_x = 10^{-7}$ in \eqref{eq:term2}. A similar results confirm our ARNT is comparable to the other algorithms.

\begin{table} \caption{Numerical results on BEC with the potential function $V_1(x,y)$} \label{tab:bec1}
	\centering
	\setlength{\tabcolsep}{1.5pt}
	\begin{tabular}{|c|cccc||cccc|}
		\hline
		
		solver & 	f  & its  & nrmG & time  & 	f  & its  & nrmG & time	\\ \hline
\multicolumn{9}{|c|}{$\beta$ =500}	\\ \hline	
& 	 \multicolumn{4}{c||}{$ \Omega =0.00$} & \multicolumn{4}{c|}{$ \Omega =0.25$} \\ \hline 
OptM  &         8.5118 &    58 &  6.6e-5 &      1.4 	&         8.5106&   103 &  9.7e-5 &     12.3  \\ \hline 
TRQH   &         8.5118 &     4(  17) &  1.5e-4 &      2.0 	 &         8.5106 &     5(  22) &  1.9e-4 &     21.9 \\ \hline 
ARNT &         8.5118 &   3(  24) & 	 1.2e-5 &	      1.5 	&         8.5106 &   4(  53) & 	 1.6e-5 &	     17.7 \\ \hline 
RTR &         8.5118 &   3(  25) & 	 1.3e-5 &	      1.5 	&         8.5106 &   3(  23) & 	 6.0e-5 &	     15.1 \\ \hline 
& 	 \multicolumn{4}{c||}{$ \Omega =0.50$} & \multicolumn{4}{c|}{$ \Omega =0.60$} \\ \hline 
OptM  &         8.0246 &   276 &  9.0e-5 &     32.3 	&         7.5890&   301 &  1.0e-4 &     19.9  \\ \hline 
TRQH   &         8.0246 &     5(  53) &  2.0e-4 &     60.7 	 &         7.5890 &     5(  60) &  1.9e-4 &     35.4 \\ \hline 
ARNT &         8.0197 &   3(  62) & 	 6.5e-5 &	     21.3 	&         7.5890 &   3(  67) & 	 5.7e-5 &	     22.1 \\ \hline 
RTR &         8.0246 &  11( 113) & 	 1.0e-4 &	     56.5 	&         7.5890 &   3(  61) & 	 5.2e-5 &	     23.8 \\ \hline 
& 	 \multicolumn{4}{c||}{$ \Omega =0.70$} & \multicolumn{4}{c|}{$ \Omega =0.80$} \\ \hline 
OptM  &         6.9731 &   340 &  1.0e-4 &     56.3 	&         6.1016&   386 &  1.0e-4 &     65.2  \\ \hline 
TRQH   &         6.9731 &     7(  55) &  2.0e-4 &     61.6 	 &         6.1016 &     5(  64) &  2.0e-4 &     83.1 \\ \hline 
ARNT &         6.9731 &  10(  99) & 	 8.7e-5 &	     44.4 	&         6.1016 &  10( 104) & 	 8.7e-5 &	     70.6 \\ \hline 
RTR &         6.9731 &  99( 118) & 	 9.3e-5 &	    234.2 	&         6.1016 &  18( 130) & 	 7.7e-5 &	    130.1 \\ \hline 
& 	 \multicolumn{4}{c||}{$ \Omega =0.90$} & \multicolumn{4}{c|}{$ \Omega =0.95$} \\ \hline 
OptM  &         4.7784 &  10000 &  1.2e-3 &    243.6 	&         3.7419&  10000 &  7.4e-4 &    241.6  \\ \hline 
TRQH   &         4.7778 &   277( 176) &  2.0e-4 &   1090.9 	 &         3.7416 &   363( 181) &  2.0e-4 &   1185.1 \\ \hline 
ARNT &         4.7777 & 147( 132) & 	 9.6e-5 &	    413.3 	&         3.7414 & 500( 147) & 	 2.6e-4 &	   1204.0 \\ \hline 
RTR &         4.7777 & 500( 147) & 	 8.5e-4 &	   1250.4 	&         3.7415 & 500( 172) & 	 9.7e-4 &	   1419.0 \\ \hline 

\multicolumn{9}{|c|}{$\beta$ =1000}	\\ \hline

& 	 \multicolumn{4}{c||}{$ \Omega =0.00$} & \multicolumn{4}{c|}{$ \Omega =0.25$} \\ \hline 
OptM  &        11.9718 &    76 &  4.6e-5 &      3.0 	&        11.9266&   358 &  9.9e-5 &     40.2  \\ \hline 
TRQH   &        11.9718 &     4(  15) &  1.0e-4 &      1.5 	 &        11.9266 &     4(  50) &  1.7e-4 &     44.3 \\ \hline 
ARNT &        11.9718 &   3(  16) & 	 3.1e-5 &	      0.9 	&        11.9266 &  15(  70) & 	 2.5e-5 &	     40.9 \\ \hline 
RTR &        11.9718 &   3(  16) & 	 3.8e-5 &	      0.8 	&        11.9266 &  15(  70) & 	 8.7e-5 &	     46.4 \\ \hline 
& 	 \multicolumn{4}{c||}{$ \Omega =0.50$} & \multicolumn{4}{c|}{$ \Omega =0.60$} \\ \hline 
OptM  &        11.1054 &   396 &  1.0e-4 &     32.6 	&        10.4392&  5524 &  1.0e-4 &    140.4  \\ \hline 
TRQH   &        11.1326 &     6(  53) &  2.0e-4 &     36.3 	 &        10.4437 &     9(  98) &  2.0e-4 &     92.8 \\ \hline 
ARNT &        11.1326 &  20(  66) & 	 5.9e-5 &	     36.8 	&        10.4392 &  20(  73) & 	 7.6e-5 &	     77.9 \\ \hline 
RTR &        11.1326 &  32(  78) & 	 5.8e-5 &	     68.9 	&        10.4392 &  93(  80) & 	 9.8e-5 &	    187.6 \\ \hline 
& 	 \multicolumn{4}{c||}{$ \Omega =0.70$} & \multicolumn{4}{c|}{$ \Omega =0.80$} \\ \hline 
OptM  &         9.5283 &   990 &  1.0e-4 &     63.7 	&         8.2627&  10000 &  5.5e-4 &    231.9  \\ \hline 
TRQH   &         9.5301 &   102( 156) &  2.0e-4 &    404.1 	 &         8.2610 &   453( 177) &  2.0e-4 &   1427.0 \\ \hline 
ARNT &         9.5301 &  60(  81) & 	 9.3e-5 &	    140.4 	&         8.2610 & 202( 105) & 	 6.7e-5 &	    412.7 \\ \hline 
RTR &         9.5301 & 293(  91) & 	 8.6e-5 &	    478.8 	&         8.2610 & 500( 113) & 	 5.5e-4 &	    972.7 \\ \hline 
& 	 \multicolumn{4}{c||}{$ \Omega =0.90$} & \multicolumn{4}{c|}{$ \Omega =0.95$} \\ \hline 
OptM  &         6.3611 &  10000 &  3.0e-3 &    230.8 	&         4.8856&  10000 &  5.2e-4 &    241.4  \\ \hline 
TRQH   &         6.3607 &   142( 170) &  2.0e-4 &    595.6 	 &         4.8831 &   172( 178) &  2.0e-4 &    708.1 \\ \hline 
ARNT &         6.3607 & 500( 110) & 	 2.8e-3 &	    931.5 	&         4.8822 & 500( 121) & 	 1.5e-3 &	   1015.8 \\ \hline 
RTR &         6.3607 & 500( 122) & 	 7.6e-4 &	   1010.8 	&         4.8823 & 500( 137) & 	 1.9e-3 &	   1103.8 \\ \hline 
		   		
	\end{tabular}
	
\end{table}

\begin{table} \caption{Numerical results on BEC with the potential function $V_2(x,y)$} \label{tab:bec2}
	\centering
	\setlength{\tabcolsep}{1.5pt}
	\begin{tabular}{|c|cccc||cccc|}
		\hline
		
		solver & 	f  & its  & nrmG & time  & 	f  & its  & nrmG & time	\\ \hline
\multicolumn{9}{|c|}{$\beta$ =500}	\\ \hline	
& 	 \multicolumn{4}{c||}{$ \Omega =0.00$} & \multicolumn{4}{c|}{$ \Omega =0.25$} \\ \hline 
OptM  &         9.3849 &   108 &  7.6e-5 &      2.8 	&         9.3849&   118 &  7.4e-5 &      5.6  \\ \hline 
TRQH   &         9.3849 &     4(  21) &  1.9e-4 &      2.6 	 &         9.3849 &     5(  17) &  1.5e-4 &      5.8 \\ \hline 
ARNT &         9.3849 &   3(  25) & 	 5.5e-5 &	      1.7 	&         9.3849 &   3(  26) & 	 4.6e-5 &	      3.6 \\ \hline 
RTR &         9.3849 &   3(  27) & 	 5.5e-5 &	      1.8 	&         9.3849 &   3(  27) & 	 5.6e-5 &	      3.7 \\ \hline 
& 	 \multicolumn{4}{c||}{$ \Omega =0.50$} & \multicolumn{4}{c|}{$ \Omega =0.60$} \\ \hline 
OptM  &         9.2053 &   142 &  9.2e-5 &     30.2 	&         9.1053&   132 &  9.8e-5 &     25.3  \\ \hline 
TRQH   &         9.2053 &     5(  23) &  1.4e-4 &     24.4 	 &         9.1053 &     5(  20) &  1.5e-4 &     20.4 \\ \hline 
ARNT &         9.2053 &   3(  27) & 	 8.4e-5 &	     19.5 	&         9.1053 &   3(  28) & 	 7.5e-5 &	     11.5 \\ \hline 
RTR &         9.2053 &   3(  29) & 	 8.3e-5 &	     20.2	&         9.1053 &   3(  30) & 	 8.5e-5 &	     19.5 \\ \hline 
& 	 \multicolumn{4}{c||}{$ \Omega =0.70$} & \multicolumn{4}{c|}{$ \Omega =0.80$} \\ \hline 
OptM  &         8.8307 &   264 &  8.3e-5 &     26.0 	&         8.4819&   374 &  8.2e-5 &     46.9  \\ \hline 
TRQH   &         8.8307 &     5(  47) &  1.8e-4 &     24.0 	 &         8.4819 &     5(  64) &  2.0e-4 &     38.7 \\ \hline 
ARNT &         8.8307 &   5(  95) & 	 3.8e-5 &	     26.0 	&         8.4819 &   3(  76) & 	 8.8e-5 &	     30.2 \\ \hline 
RTR &         8.8307 &   3(  87) & 	 7.6e-5 &	     47.2 	&         8.4819 &   3(  94) & 	 7.4e-5 &	    26.4 \\ \hline 
& 	 \multicolumn{4}{c||}{$ \Omega =0.90$} & \multicolumn{4}{c|}{$ \Omega =0.95$} \\ \hline 
OptM  &         8.0659 &   426 &  1.0e-4 &     75.2 	&         7.7455&  9508 &  9.9e-5 &    244.4  \\ \hline 
TRQH   &         8.0659 &     5(  94) &  1.4e-4 &    124.1 	 &         7.7455 &    21( 155) &  2.0e-4 &    254.8 \\ \hline 
ARNT &         8.0659 &   3(  99) & 	 9.0e-5 &	     56.4 	&         7.7455 &  30( 192) & 	 2.3e-5 &	    171.7 \\ \hline 
RTR &         8.0659 &   3( 108) & 	 8.7e-5 &	    107.6 	&         7.7455 &  20( 270) & 	 9.7e-5 &	    257.4 \\ \hline 
 
\multicolumn{9}{|c|}{$\beta$ =1000}	\\ \hline

& 	 \multicolumn{4}{c||}{$ \Omega =0.00$} & \multicolumn{4}{c|}{$ \Omega =0.25$} \\ \hline 
OptM  &        14.9351 &   158 &  9.0e-5 &      5.1 	&        14.9351&   113 &  9.5e-5 &      9.2  \\ \hline 
TRQH   &        14.9351 &     5(  22) &  2.0e-4 &      3.1 	 &        14.9667 &    32( 160) &  2.0e-4 &    133.1 \\ \hline 
ARNT &        14.9351 &   3(  33) & 	 6.6e-5 &	      2.1 	&        14.9667 &  41( 131) & 	 7.4e-5 &	    105.2 \\ \hline 
RTR &        14.9351 &   3(  33) & 	 6.1e-5 &	      2.0 	&        14.9667 &  37( 137) & 	 7.9e-5 &	    118.5 \\ \hline 
& 	 \multicolumn{4}{c||}{$ \Omega =0.50$} & \multicolumn{4}{c|}{$ \Omega =0.60$} \\ \hline 
OptM  &        14.7167 &  1261 &  9.9e-5 &     68.6 	&        14.4704&  1128 &  1.0e-4 &     38.6  \\ \hline 
TRQH   &        14.7167 &    11( 123) &  2.0e-4 &     73.3 	 &        14.6167 &    13(  65) &  1.5e-4 &     58.1 \\ \hline 
ARNT &        14.7167 &  17( 127) & 	 8.0e-5 &	     66.8 	&        14.6167 &   7( 112) & 	 7.0e-5 &	     43.6 \\ \hline 
RTR &        14.7167 &  13( 136) & 	 7.3e-5 &	     72.2 	&        14.6167 &   3(  33) & 	 6.2e-5 &	     42.2 \\ \hline 
& 	 \multicolumn{4}{c||}{$ \Omega =0.70$} & \multicolumn{4}{c|}{$ \Omega =0.80$} \\ \hline 
OptM  &        14.2813 &   719 &  1.0e-4 &     47.8 	&        13.8647&  4382 &  1.0e-4 &    118.3  \\ \hline 
TRQH   &        14.5167 &     7(  31) &  1.9e-4 &     41.5 	 &        13.6368 &    42( 169) &  2.0e-4 &    283.4 \\ \hline 
ARNT &        14.5167 &   5( 104) & 	 9.9e-5 &	     38.2 	&        13.6561 &  39( 138) & 	 5.5e-5 &	    133.5 \\ \hline 
RTR &        14.5167 &   3(  33) & 	 9.3e-5 &	     33.2	&        13.6561 &  29( 153) & 	 9.6e-5 &	    144.9 \\ \hline 
& 	 \multicolumn{4}{c||}{$ \Omega =0.90$} & \multicolumn{4}{c|}{$ \Omega =0.95$} \\ \hline 
OptM  &        13.3733 &  5004 &  1.0e-4 &    166.9 	&        12.8180&  10000 &  3.2e-3 &    270.8  \\ \hline 
TRQH   &        13.3733 &     6( 108) &  1.9e-4 &    117.3 	 &        12.8048 &   423( 143) &  1.8e-4 &   1153.7 \\ \hline 
ARNT &        13.3733 &   8( 166) & 	 3.9e-5 &	     68.1 	&        12.8180 &  53( 167) & 	 6.8e-5 &	    191.8 \\ \hline 
RTR &        13.3733 &  12( 199) & 	 4.1e-5 &	    93.5 	&        12.8180 &  66( 250) & 	 9.8e-5 &	    339.9 \\ \hline

	\end{tabular}
	
\end{table}

\subsection{Low-Rank Matrix Completion}
 Given a partially observed matrix $A \in \R^{m\times n}$, we want to find the
lowest-rank matrix to fit $A$ on the known elements. This problem can be formulated as follows:
\bee \label{prob:lrmc}  \min_{X \in \R^{m\times n}}~f(X):= \half\|P_{\Omega}(X) - A\|_F^2 \quad
\st \quad  X \in \{ X \in \R^{m\times n}: \mathrm{rank}(X) = k \} \eee
where $P_\Omega : \R^{m\times n} \rightarrow \R^{m \times n}$, $P_\Omega(X)_{i,j} := X_{i,j}$ if $(i,j) \in \Omega$ and $P_\Omega(X)_{i,j} := 0$ if $(i,j) \notin \Omega$, is the projection onto $\Omega$ and $\Omega$ is a subset of $\{1,\ldots, m\} \times \{1,\ldots,n\}$. More details can be found in \cite{vandereycken2013low}.
%If the lowest rank is known, say $k$, the authors in \cite{vandereycken2013low} converted problem \eqref{prob:lrmc1} into the following equivalent form.  Since $\M_k$ is a smooth manifold. This problem can be solved by Riemannian optimization algorithms.

Similar to \cite{vandereycken2013low}, we construct random numerical examples as
follows. We first take two Gaussian random matrices $A_L, A_R \in \R^{n \times
k}$, then uniformly sample the index set $\Omega$
 for a given cardinality and set the matrix $A:= P_{\Omega}(A_LA_R^\top)$. Since
 the degrees of freedom in a nonsymmetric matrix of rank $k$ is given $k(2n-k)$,  we
 define the ratio $r_S = {(k(2n-k))^{-1}}|\Omega|$. % we will test the influence of the different ratios.
In this example, we only penalize $x - x_k$ on the known set $\Omega$ in the implementation of ARNT to reduce the computational costs. (I.e., the penalization term in the subproblem \cref{eq:mQ} is set to $\sigma_k \|P_{\Omega}(x- x_k)\|^2$).
 In the Tables \ref{num:lrmc-n}, \ref{num:lrmc-k} and \ref{num:lrmc-ratio}, we can see that ARNT and RTR perform better than GBB regardless whether the dimension $n$ and
rank $k$ are large or small. %When the sample measurement ratio $r_S$ is close to zero,
We often observe that ARNT tends to outperform RTR when  negative curvature is
 encountered. % which seems to be handled more efficiently by ARNT. 
 %Let us note that when $r_S$ is large (i.e., $r_S > 1$), then  ARNT and RTR perform comparably since in these cases usually  no negative curvature is encountered. 
%in this case the negative curvature is often encountered. This leads our ARNT algorithm to be practical in real problems when the measurements can not be easily obtained.

\begin{table} \caption{Numerical results on low rank matrix completion with
  the fixed $k=10, r_S = 0.8$ but different $n$.} \label{num:lrmc-n}
	\centering
	\setlength{\tabcolsep}{1.5pt}
	\begin{tabular}{|c|ccc|ccc|ccc|}
		\hline
		& 	 \multicolumn{3}{c|}{GBB} & 	 \multicolumn{3}{c|}{ARNT} & 	 \multicolumn{3}{c|}{RTR}\\ \hline
       $n$	 & 	 its & 	 nrmG &	 time & 	 its & 	 nrmG &	 time & 	 its & 	 nrmG &	 time\\ \hline
       1000 & 	 603 & 	 5.1e-7 &	 12.5 & 	 6(84) & 	 3.4e-7 &	 7.7 & 	 8(91) & 	 6.6e-7 &	 8.2\\ \hline 
       2000 & 	 570 & 	 9.2e-7 &	 43.9 & 	 5(72) & 	 8.9e-7 &	 23.6 & 	 8(86) & 	 6.2e-7 &	 28.2\\ \hline 
       4000 & 	 671 & 	 9.7e-7 &	 179.8 & 	 6(82) & 	 4.6e-7 &	 94.8 & 	 9(85) & 	 2.0e-7 &	 104.8\\ \hline 
       8000 & 	 666 & 	 9.8e-7 &	 694.2 & 	 5(104) & 	 5.2e-7 &	 320.1 & 	 8(130) & 	 5.4e-7 &	 394.5\\ \hline  
				
	\end{tabular}
	
\end{table}

\begin{table} \caption{Numerical results on low rank matrix completion with
  fixed $n=4000, r_S = 0.95$ but different $k$.} \label{num:lrmc-k}
	\centering
	\setlength{\tabcolsep}{1.5pt}
	\begin{tabular}{|c|ccc|ccc|ccc|ccc|}
		\hline
		
		& 	 \multicolumn{3}{c|}{GBB} & 	 \multicolumn{3}{c|}{ARNT} & 	 \multicolumn{3}{c|}{RTR}\\ \hline
       $k$ 	 & 	 its & 	 nrmG &	 time & 	 its & 	 nrmG &	 time & 	 its & 	 nrmG &	 time\\ \hline
        10 & 	 5252 & 	 1.0e-6 &	 1415.9 & 	 13(133) & 	 7.4e-7 &	 392.1 & 	 12(236) & 	 4.4e-7 &	 438.2\\ \hline 
        20 & 	 2126 & 	 1.0e-6 &	 600.8 & 	 7(125) & 	 3.9e-7 &	 269.5 & 	 9(195) & 	 2.3e-7 &	 315.9\\ \hline 
        30 & 	 1488 & 	 1.0e-6 &	 438.8 & 	 6(132) & 	 3.1e-7 &	 255.2 & 	 9(214) & 	 2.6e-7 &	 329.9\\ \hline 
        40 & 	 1010 & 	 9.3e-7 &	 311.4 & 	 5(103) & 	 1.1e-7 &	 220.5 & 	 5(103) & 	 1.1e-7 &	 219.4\\ \hline 
        50 & 	 1494 & 	 7.9e-7 &	 477.1 & 	 4(103) & 	 1.5e-7 &	 273.8 & 	 4(103) & 	 1.6e-7 &	 272.5\\ \hline 
        60 & 	 1398 & 	 9.9e-7 &	 480.4 & 	 4(110) & 	 5.7e-7 &	 313.3 & 	 4(114) & 	 5.7e-7 &	 315.2\\ \hline  
				
	\end{tabular}
	
\end{table}

\begin{table} \caption{Numerical results on low rank matrix completion with
  fixed $n=8000, k = 10$ but different $r_S$.} \label{num:lrmc-ratio}
	\centering
	\setlength{\tabcolsep}{1.5pt}
	\begin{tabular}{|c|ccc|ccc|ccc|ccc|}
		\hline
		
		& 	 \multicolumn{3}{c|}{GBB} & 	 \multicolumn{3}{c|}{ARNT} & 	 \multicolumn{3}{c|}{RTR}\\ \hline
       $r_S$ 	 & 	 its & 	 nrmG &	 time & 	 its & 	 nrmG &	 time & 	 its & 	 nrmG &	 time\\ \hline
       0.1 & 	 86 & 	 3.7e-7 &	 88.1 & 	 3(11) & 	 4.4e-7 &	 54.7 & 	 3(11) & 	 4.3e-7 &	 53.2\\ \hline 
       0.2 & 	 89 & 	 8.6e-7 &	 93.9 & 	 3(14) & 	 3.4e-7 &	 55.5 & 	 3(14) & 	 3.4e-7 &	 54.0\\ \hline 
       0.3 & 	 117 & 	 9.5e-7 &	 119.7 & 	 3(14) & 	 4.2e-7 &	 67.3 & 	 3(14) & 	 4.2e-7 &	 66.1\\ \hline 
       0.5 & 	 173 & 	 8.5e-7 &	 178.8 & 	 3(18) & 	 7.0e-7 &	 111.4 & 	 3(18) & 	 7.0e-7 &	 109.7\\ \hline 
       0.8 & 	 666 & 	 9.8e-7 &	 700.2 & 	 5(104) & 	 5.2e-7 &	 318.7 & 	 8(130) & 	 5.4e-7 &	 388.8\\ \hline

	\end{tabular}
	
\end{table}

\section{Conclusions}
In this paper, we propose a regularized Newton method for optimization problems on
Riemannian manifolds. We use a second-order approximation of the objective
function in the Euclidean space to form a sequence of quadratic subproblems
while keeping the manifold constraints. A modified Newton method is then developed and analyzed to solve the resulting subproblems. Based on a Steihaug-type CG method, we construct a specific search direction that can use negative curvature information of the Riemannian Hessian. We show that our method enjoys favorable convergence properties and converges with a locally superlinear rate. Numerical experiments are
performed on the nearest correlation matrix estimation, Kohn-Sham total energy
minimization, BEC, and low-rank matrix completion problems. The comparisons illustrate that our
proposed method is promising. In particular, it can often reach a certain level of accuracy faster than other
state-of-the-art algorithms. Our algorithm performs comparable to the Riemannian trust-region (RTR) method and usually achieves a better convergence rate once negative curvature is encountered. We should point out that our proposed algorithm can be further improved if a more specialized and efficient solver for the inner subproblem is available. % remark that the truncated Newton method is a general method to solve the subproblem. Once we develop other more efficient algorithms for subproblem, the faster convergence will be expected. 

\section*{Acknowledgements} We would like to thank Bo Jiang for the helpful discussion on optimization with orthogonality constraints.  %\setcitestyle{numbers}
%\bibsep=0.1cm
%\bibliographystyle{plainnat}
\bibliographystyle{siamplain}
\bibliography{optimization}

\end{document}